\documentclass{article}
\usepackage{amsmath}
\usepackage{amssymb, amscd, amsthm, amsfonts, latexsym}
\usepackage{enumerate}
\usepackage[dvips]{graphicx,color}

\newcommand{\no}{\noindent}
\newtheorem{thm}{Theorem}[section]
\newtheorem{prop}[thm]{Proposition}
\newtheorem{cor}[thm]{Corollary}
\newtheorem{rem}[thm]{Remark}
\newtheorem{defi}[thm]{Definition}
\newtheorem{ex}[thm]{Example}
\newtheorem{lem}[thm]{Lemma}

\newtheorem{prob}[thm]{Problem}
\newtheorem{ass}[thm]{Assumption}
\numberwithin{equation}{section}

\newcommand{\R}{\mathbb{R}}

\newcommand{\ds}{\displaystyle}

\newcommand{\sm}{\setminus}
\newcommand{\pd}{\partial}
\newcommand{\al}{\alpha}
\newcommand{\be}{\beta}

\newcommand{\de}{\delta}
\newcommand{\De}{\Delta}
\newcommand{\ep}{\varepsilon}
\newcommand{\la}{\lambda}
\newcommand{\f}{\varphi}

\newcommand{\Ome}{\Omega}

\renewcommand{\(}{\left(}
\renewcommand{\)}{\right)}
\renewcommand{\lvert}{\left\vert}
\renewcommand{\rvert}{\right\vert}

\DeclareMathOperator{\diam}{diam}
\DeclareMathOperator{\conv}{conv}

\DeclareMathOperator{\dist}{dist}
\DeclareMathOperator{\Vol}{Vol}

\DeclareMathOperator{\Refl}{Refl}
\DeclareMathOperator{\Span}{Span}

\DeclareMathOperator{\Rot}{Rot}
\DeclareMathOperator{\Ker}{Ker}

\def\c#1{\overset{\mbox{\tiny $\circ$}}{#1}} 

\begin{document}

\title{\bf Geometric estimation of a potential and cone conditions of a body}

\author{Shigehiro Sakata}

\date{\today}

\maketitle

\begin{abstract}
We investigate a potential obtained as the convolution of a radially symmetric function and the characteristic function of a body (the closure of a bonded open set) with exterior cones. In order to restrict the location of a maximizer of the potential into a smaller closed region contained in the interior of the body, we give an estimate of the potential using the exterior cones of the body. Moreover, we apply the result to the Poisson integral for the upper half space.\\

\no{\it Keywords and phrases}. Hot spot, Poisson integral, solid angle, illuminating center, Riesz potential, Hadamard finite part, renormalization, $r^{\al -m}$-potential, minimal unfolded region, heart, cone condition.\\
\no 2010 {\it Mathematics Subject Classification}: 31B25, 35B38, 35B50, 51M16, 52A40.
\end{abstract}
\section{Introduction}
Let $\Ome$ be a body (the closure of a bounded open set) in $\R^m$. We consider a potential of the form
\begin{equation}\label{K}
K_\Ome (x,t)=\int_\Ome k\( \lvert x-\xi \rvert ,t\)d\xi ,\ x \in \R^m ,\ t>0,
\end{equation}
and investigate its spatial maximizer.

When $k(r,t)$ is given by the Gauss kernel, the potential $K_\Ome(x,t)$ is the solution of the Cauchy problem for the heat equation with initial datum $\chi_\Ome$,
\begin{equation}\label{W}
W_\Ome (x,t)=\frac{1}{\( 4\pi t\)^{m/2}} \int_\Ome \exp \( -\frac{\lvert x-\xi \rvert^2}{4t}\) d\xi ,\ x\in \R^m ,\ t>0 .
\end{equation}
A spatial maximizer of $W_\Ome$ is called a {\it hot spot} of $\Ome$ at time $t$.

In \cite{CK}, Chavel and Karp showed that $\Ome$ has a hot spot for each $t$, that any hot spot belongs to the convex hull of $\Ome$, and that the set of hot spots converges to the one-point set of the centroid (center of mass) of $\Ome$ as $t$ goes to infinity with respect to the Hausdorff distance. Furthermore, calculating the Hessian of $W_\Ome (\cdot ,t) :\R^m \to \R$, in \cite{JS}, Jimbo and Sakaguchi indicated that $\Ome$ has a unique hot spot whenever $t \geq (\diam \Ome)^2/2$. Roughly speaking, the large-time behavior of hot spots was studied in \cite{CK, JS}. (To tell the truth, in \cite{CK, JS}, the above properties of hot spots were shown for a non-zero non-negative bounded compactly supported initial datum. But, in this paper, we are interested in the case where the initial datum is given by the characteristic function of a body.)

In contrast, in \cite{KP}, Karp and Peyerimhoff gave a geometric heat comparison criteria and investigated the small-time behavior of hot spots. Roughly speaking, they compared two heat flows for two points in two different bodies by using the distance functions from the complements and showed that any sequence of hot spots of $\Ome$ at time $t_\ell$ converges to an {\it incenter} of $\Ome$ as $t_\ell$ tends to zero. Let us review their exact statement as below: Let $X$ and $Y$ be bodies in $\R^m$; Fix two constants $R> S\geq 0$; Let $X'= \{ x \in X \vert \dist (x, X^c) \geq R\}$, and $Y' = \{ y \in \R^m \vert \dist (y,Y^c) \leq S\}$; Then,  we can choose a small time $\tau$ such that if $0<t<\tau$, then, for any $x \in X'$ and $y \in Y'$, we have $H_X (x,t) > H_Y(y,t)$; Taking $X=Y=\Ome$ and $R=R_\infty (\Ome )= \max \dist (\xi ,\Ome^c)$ (the {\it inradius}), we can conclude that, for any decreasing sequence $\{ t_\ell\}$ with zero limiting value and any hot spot $h(t_\ell)$ of $\Ome$ at time $t_\ell$, the distance between $h(t_\ell)$ and the set of incenters $\mathcal{I}_\Ome = \{ \xi \in \Ome \vert \dist (\xi ,\Ome^c ) =R_\infty (\Ome )\}$ tends to zero as $\ell$ goes to infinity. (To tell the truth, they investigated the above comparison theorem in Riemannian manifolds. But, in this paper, we are interested in Euclidean case.) We also refer to \cite[pp. 2--3]{MS} for the small-time behavior of hot spots.

On the other hand, in \cite{Sak1}, for the kernel $k$ in \eqref{K}, the author gave a sufficient condition implying the results shown in \cite{CK, JS}. As a by-product, for example, his sufficient condition can be applied to the {\it Poisson integral} for the upper half space,
\begin{equation}\label{P}
P_\Ome (x,h) =\frac{2h}{\sigma_m \(S^m \)} \int_\Ome \( \lvert x -\xi \rvert^2 +h^2 \)^{-(m+1)/2} d\xi ,\ x \in \R^m ,\ h>0 ,
\end{equation}
where $\sigma_m$ denotes the $m$-dimensional spherical Lebesgue measure. Precisely, his sufficient condition implies that the function $P_\Ome (\cdot ,h) :\R^m \to \R$ has a maximizer for each $h$, that any maximizer of $P_\Ome (\cdot ,h)$ belongs to the convex hull of $\Ome$, that the set of maximizers of $P_\Ome (\cdot ,h)$ converges to the one-point set of the centroid of $\Ome$ as $h$ goes to infinity with respect to the Hausdorff distance, and that $P_\Ome (\cdot ,h)$ has a unique maximizer whenever $h\geq \sqrt{m+2} \diam \Ome$.

Here, we remark that the Poisson integral $P_\Ome$ satisfies the Laplace equation for the upper half space, that is, 
\begin{equation}
\( \sum_{j=1}^{m} \frac{\pd^2}{\pd x_j^2} + \frac{\pd^2}{\pd h^2} \) P_\Ome (x,h) =0,\ x \in \R^m ,\ h>0,
\end{equation}
and we have the boundary condition
\begin{equation}
\lim_{h \to 0^+} P_\Ome (x,h) = \chi_\Ome (x) ,\ x \in \R^m \sm \pd \Ome .
\end{equation}

In order to understand the geometric meaning of the author's results on maximizers of $P_\Ome (\cdot ,h)$, let $A_\Ome = \sigma_m (S^m) P_\Ome /2$. The function $A_\Ome$ is obtained in the following manner: Let $x$ be a point in $\R^m$, and $h$ a positive constant; Define the map
\begin{equation}
p_{(x,h)} : \Ome \times \{ 0 \} \ni (\xi ,0) \mapsto \frac{(\xi ,0) -(x,h)}{\lvert (\xi ,0) -(x,h) \rvert} \in S^m ;
\end{equation}
The {\it solid angle} of $\Ome$ at $(x,h)$ is defined as the $m$-dimensional spherical Lebesgue measure of the image $p_{(x,h)} (\Ome )$ (see Figure \ref{angle}), and direct calculation shows
\begin{equation}
\sigma_m \( p_{(x,h)} (\Ome )\) = h \int_\Ome  \( \lvert x -\xi \rvert^2 +h^2 \)^{-(m+1)/2} d\xi = A_\Ome (x,h) .
\end{equation}
In \cite{Sh}, the solid angle of $\Ome$ as $(x,h)$ was regarded as the ``brightness'' of $\Ome$ having a light source at $(x,h)$. We call a maximizer of $A_\Ome (\cdot ,h)$ an {\it illuminating center} of $\Ome$ of height $h$. Thus, the properties of $P_\Ome$ shown in \cite{Sak1} are understood as the large-height behavior of illuminating centers. In other words, it was shown that the large-parameter behavior of spatial maximizers of $P_\Ome$ is similar to that of $W_\Ome$.
\begin{figure}[hbtp]
\centering
\scalebox{0.28}{\includegraphics[clip]{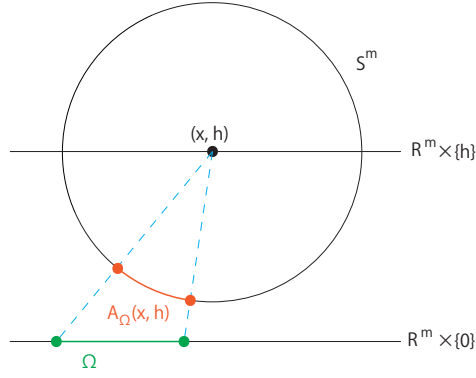}}
\caption{The solid angle of $\Ome$ at $(x,h)$}
\label{angle}
\end{figure}

From such backgrounds, in this paper, in order to compare small-parameter behavior of spatial maximizers of $P_\Ome$ and $W_\Ome$, we mainly investigate the small-height behavior of illuminating centers. Informal computation shows
\begin{equation}
\frac{A_\Ome (x,h)}{h} 
=\int_\Ome  \( \lvert x -\xi \rvert^2 +h^2 \)^{-(m+1)/2} d\xi  
\to \int_\Ome  \lvert x -\xi \rvert^{-(m+1)} d\xi
\end{equation}
as $h$ tends to $0^+$. But the right hand side diverges whenever $x$ is in $\Ome$. Then, for a point $x$ in the interior of $\Ome$, let us consider its {\it Hadamard finite part}, 
\begin{align}
\label{V1}
V_\Ome^{(-1)} (x) 
&=\lim_{\ep \to 0^+} \( \int_{\Ome \sm B_\ep (x)} \lvert x-\xi \rvert^{-(m+1)} - \frac{\sigma_{m-1} \( S^{m-1} \)}{\ep} \) \\
\label{V2} 
&=\int_{\Ome \sm B_\ep (x)} \lvert x-\xi \rvert^{-(m+1)} - \frac{\sigma_{m-1} \( S^{m-1} \)}{\ep} .
\end{align}
Here, we remark that the latter equality \eqref{V2} holds whenever $0< \ep < \dist (x, \Ome^c)$ (see Proposition \ref{without_limit}). 

It is expected that any sequence of illuminating centers of height $h_\ell$ converges to a maximizer of $V_\Ome^{(-1)}$ as $h_\ell$ tends to zero. This expectation comes from the following procedure: Let $\ep >0$ be small enough; Suppose that, for any small enough $h>0$, any illuminating center is at least $\ep$ away from the boundary of $\Ome$; Since the Poisson kernel is radially symmetric, the solid angle of $B_\ep (x)$ at $(x,h)$ depends only on $\ep$ and $h$; Decomposing the solid angle function as $A_\Ome (x,h)= A_{\Ome \sm B_\ep (x)}(x,h)+A_{B_\ep (x)}(x,h)$, a point $c(h)$ is an illuminating center if and only if it is a maximizer of $A_{\Ome \sm B_\ep (\cdot )}(\cdot ,h)$; As $h$ tends to zero, the kernel $( \vert x-\xi \vert^2 +h^2)^{-(m+1)/2}$ converges to $\vert x-\xi \vert^{-(m+1)}$ uniformly for $\xi$ in $\Ome \sm B_\ep (x)$; Roughly speaking, if the height parameter $h$ is small enough, then we have
\begin{equation}
\frac{A_\Ome(x,h)}{h} 
= \frac{A_{\Ome \sm B_\ep (x)}(x,h)}{h} +\frac{A_{B_\ep (x)}(x,h)}{h} 
\approx V_\Ome^{(-1)}(x) +\frac{\sigma_{m-1} \( S^{m-1} \)}{\ep} +\frac{A_{B_\ep (x)}(x,h)}{h} 
\end{equation}
for any point $x$ in the interior of $\Ome$ with $\dist (x,\Ome^c) >\ep$.

In order to formulate the above procedure, we have to give a closed subset in the interior of $\Ome$ such that it contains all the illuminating centers for any small enough $h>0$. This is because we can use the expression \eqref{V2} of the potential $V_\Ome^{(-1)}$ only in the interior of $\Ome$. Namely, in \eqref{V2}, we want to take a uniform $\ep$ for illuminating centers of any small enough height and maximizers of $V_\Ome^{(-1)}$. 

We refer to \cite{BMS, BM, O3, O4, Sak1, Sak2} for the study on the location of maximizers of a potential. Some authors tried to restrict the location of maximizers of a potential into a smaller region. Using the {\it moving plane argument} (\cite{GNN, Ser}), all the maximizers of a potential with a radially symmetric strictly decreasing kernel are contained in the {\it minimal unfolded region} of $\Ome$. (The minimal unfolded region is sometimes called the {\it heart}.) But, in general, the minimal unfolded region of $\Ome$ is not contained in the interior of $\Ome$ (see Example \ref{uf_triangle}).

In this paper, assuming the {\it uniform interior cone condition} for the complement of the body $\Ome$ (see Definition \ref{interior}) and taking the following three steps, we formulate the above procedure:
\begin{description}
\item[Step 1.] We give a constant $0< \tilde{R} < R_\infty (\Ome)$ such that, for any $x \in \Ome$ with $\dist (x,\Ome^c) = R_\infty (\Ome )$ and $y \in \Ome$ with $\dist (y,\Ome^c) \leq \tilde{R}$, we have $V_\Ome^{(-1)} (y) < V_\Ome^{(-1)} (x)$. Namely, any maximizer of $V_\Ome^{(-1)}$ belongs to the inner-parallel body of $\Ome$ of radius $\tilde{R}$.
\item[Step 2.] For any constant $0<b<1$, there exits a positive $h_0$ such that if $0<h<h_0$, then, for any $x \in \Ome$ with $\dist (x,\Ome^c) = R_\infty (\Ome )$ and $y \in \R^m$ with $\dist (y,\Ome^c) \leq b\tilde{R}$, we have $A_\Ome (y,h) < A_\Ome (x,h)$. Namely, if $h$ is sufficiently small, then any illuminating center belongs to the inner-parallel body of $\Ome$ of radius $b\tilde{R}$.
\item[Step 3.] The limit point of any illuminating center of height $h_\ell$ must be a maximizer of $V_\Ome^{(-1)}$.
\end{description}

Moreover, the above argument can be extended to a general case. Precisely, we give the same estimate as in the first step to the Hadamard finite part of the {\it Riesz potential},
\begin{equation}
V_\Ome^{(\al )}(x)=
\begin{cases}
\ds \lim_{\ep \to 0^+} \( \int_{\Ome \sm B_\ep (x)} \lvert x-\xi \rvert^{\al-m}d\xi -\frac{\sigma_{m-1} \( S^{m-1}\)}{-\al} \ep^\al \) &(\al <0) ,\\
\ds \lim_{\ep \to 0^+} \( \int_{\Ome \sm B_\ep (x)} \lvert x-\xi \rvert^{-m}d\xi -\sigma_{m-1} \( S^{m-1}\) \log \frac{1}{\ep} \) &(\al =0) .
\end{cases}
\end{equation}
Also, we give the same estimate as in the second step to a potential of the form \eqref{K}. In other words, our main result in this paper is the estimate of a potential like the second step, and, as its by-product, we derive the small-height behavior of illuminating centers.\\

Throughout this paper, $\c{X}$, $\bar{X}$, $X^c$, $R_\infty (X)$ and $\diam X$ denote the interior, the closure, the complement, the inradius and the diameter of a set $X$ in $\R^m$, respectively. For a set $X$ in $\R^m$ and a positive constant $\rho$, the symbol $X \sim \rho B^m$ denotes the inner-parallel body of $X$ of radius $\rho$, that is, $X \sim \rho B^m =\{ x \in X \vert \dist (x ,X^c ) \geq \rho\}$. We denote the $m$-dimensional closed ball of radius $\rho$ and centered at $x$ by $B_\rho (x) =\rho B^m + x$. We denote a point $x$ in $\R^m$ by $x = \( x_1,\ldots,x_m\)$. The $N$-dimensional spherical Lebesgue measure is denoted by $\sigma_N$. In particular, the symbol $\sigma$ is used in the case of $N=m-1$, for short.\\

\no{\bf Acknowledgements.}
The author would like to express his deep gratitude to Professor Jun O'Hara, Professor Kazushi Yoshitomi and Professor Hiroaki Aikawa. O'Hara gave him kind advice throughout writing this paper. Yoshitomi informed him of some cone conditions. Aikawa informed him of the proof of Lemma \ref{cone_condition} and Remark \ref{rem_cone}.
\section{Preliminaries}
\subsection{Cone conditions}
Let us prepare the cone conditions which are related to the complexity of the boundary of a body. Throughout this paper, we understand that $C$ is an open cone of vertex $x$, axis direction $v$, aperture angle $\kappa$ and height $\de$ if $C$ is given as 
\begin{equation}
C= \left\{ x +\rho Rv \lvert 0 < \rho < \de,\ R \in SO(m),\ Rv \cdot v > \cos \frac{\kappa}{2} \right\} \right. .
\end{equation}

\begin{defi}\label{interior}
{\rm An open set $U$ in $\R^m$ satisfies the {\it uniform interior cone condition} if there exists an open cone $C$ in $\R^m$ such that, for each point $x \in U$, we can take an open cone of vertex $x$ contained in $U$ and congruent to $C$.}
\end{defi}

\begin{defi}\label{inner}
{\rm An open set $U$ in $\R^m$ satisfies the {\it uniform boundary inner cone condition} if there exists an open cone $C$ in $\R^m$ such that, for each point $x \in \pd U$, we can take an open cone of vertex $x$ contained in $U$ and congruent to $C$.}
\end{defi}

The proof of the following Lemma is due to Hiroaki Aikawa.

\begin{lem}\label{cone_condition}
Let $U$ be an open set in $\R^m$. If $U$ satisfies the uniform interior cone condition for an open cone $C$ of aperture angle $\kappa$ and height $\de$, then it also satisfies the uniform boundary inner cone condition for the cone $C$.
\end{lem}

\begin{proof}
Fix a point $x$ on the boundary of $U$. For each natural number $n$, we take a point $\xi^n$ from $B_{1/n} (x) \cap U$. Thanks to the uniform interior cone condition of $U$, we can take an open cone $C(\xi^n)$ of vertex $\xi^n$ contained in $U$ and congruent to $C$. Let $v^n$ be the axis direction of $C(\xi^n)$. Since the unit sphere $S^{m-1}$ is compact in $\R^m$, we may assume that the sequence $\{ v^n \}$ converges to a direction $v$. 

Let $C(x)$ be the open cone of vertex $x$ and axis direction $v$ congruent to $C$. We show that $C(x)$ is contained in $U$. Suppose that $C(x)$ is not contained in $U$. We take a point from $C(x) \cap U^c$. The point can be expressed as $x+ \rho Rv$ for some $0< \rho <\de$ and rotation matrix $R$ with $Rv \cdot v > \cos (\kappa /2)$. We remark that the point $\xi^n +\rho R v^n$ is in $C(\xi^n)$. Since $C(\xi^n)$ is contained in $U$ for any $n$, we have
\[
\lvert \( x+\rho Rv\) -\( \xi^n +\rho R v^n \) \rvert 
\geq \dist \( \xi^n +\rho R v^n, U^c \) 
\geq \min \left\{ \de-\rho ,\ \rho \sin \( \frac{\kappa}{2} -\theta \) \right\} ,
\]
where $\theta = \arccos ( Rv \cdot v )$. On the other hand, for any large enough $n$, 
\[
\lvert \( x+\rho Rv\) -\( \xi^n +\rho R v^n \) \rvert  
< \frac{1}{2} \min \left\{ \de-\rho ,\ \rho \sin \( \frac{\kappa}{2} -\theta \) \right\} ,
\]
which is a contradiction.
\end{proof}

\begin{rem}
{\rm Let $\Ome$ be a body (the closure of a bounded open set) in $\R^m$. Regarding a half space as a cone of aperture angle $\pi$ and height $+\infty$, $\Ome$ is convex if and only if the complement of $\Ome$ satisfies the uniform boundary inner cone condition of aperture angle $\pi$ and height $+\infty$.}
\end{rem}

\begin{rem}\label{rem_cone}
{\rm In Lemma \ref{cone_condition}, for an open set $U$, we showed that the uniform interior cone condition implies the uniform boundary inner cone condition. We remark that the converse statement does not always hold. For example, let us consider an open unit disc and remove a cusp from the disc near the center. Let $U$ be such an open set. Then, $U$ satisfies the uniform boundary inner cone condition but not the uniform interior cone condition. 

The author would like to express his gratitude to Professor Hiroaki Aikawa for informing him of this example.}
\end{rem}

\begin{prob}
Let $U$ be the interior of a body in $\R^m$. Does the uniform boundary inner cone condition of $U$ imply the uniform interior cone condition of $U$?
\end{prob}
\subsection{Renormalization of the Riesz potential}
Let $\Ome$ be a body (the closure of a bounded open set) in $\R^m$. We consider the {\it Riesz potential} of $\Ome$ of order $0<\al <m$,
\begin{equation}
V_\Ome^{(\al )} (x) = \int_\Ome \lvert x- \xi \rvert^{\al -m} d\xi ,\ x \in \R^m .
\end{equation}
We remark that if $\al \leq 0$, then the above integral diverges for any interior point $x$ of $\Ome$. In \cite{O3}, O'Hara extended the potential $V_\Ome^{(\al )}$ to the case of $\al \leq 0$ by using the same renormalizing process as in the definition of his energy of knots introduced in \cite{O1, O2}. Precisely, for $\al \leq 0$ and $x \in \c{\Ome}$, define the {\it renormalization} of the Riesz potential
\begin{equation}
V_\Ome^{(\al )}(x)=
\begin{cases}
\ds \lim_{\ep \to 0^+} \( \int_{\Ome \sm B_\ep (x)} \lvert x-\xi \rvert^{\al-m}d\xi -\frac{\sigma \( S^{m-1}\)}{-\al} \ep^\al \) &(\al <0) ,\\
\ds \lim_{\ep \to 0^+} \( \int_{\Ome \sm B_\ep (x)} \lvert x-\xi \rvert^{-m}d\xi -\sigma \( S^{m-1}\) \log \frac{1}{\ep} \) &(\al =0) ,
\end{cases}
\end{equation}
and we call it the {\it $r^{\al -m}$-potential} of order $\al$ in what follows. Here, for $\al \leq 0$ and $x \in \Ome^c$, we define the potential $V_\Ome^{(\al)}(x)$ as the usual Riesz potential, that is,
\begin{equation}\label{complementV}
V_\Ome^{(\al )} (x) = \int_\Ome \lvert x- \xi \rvert^{\al -m} d\xi ,\ \al \leq 0,\ x \in \Ome^c.
\end{equation}

Let us prepare some terminologies and properties of $V_\Ome^{(\al )}$ from \cite{O3}.

\begin{prop}[{\cite[Proposition 2.5]{O3}}]\label{without_limit}
Let $\Ome$ be a body in $\R^m$. For $\al \leq 0$ and $x \in \c{\Ome}$, we have
\[
V_\Ome^{(\al )}(x)=
\begin{cases}
\ds \int_{\Ome \sm B_\ep (x)} \lvert x-\xi \rvert^{\al-m}d\xi -\frac{\sigma \( S^{m-1}\)}{-\al} \ep^\al  &(\al <0) ,\\
\ds \int_{\Ome \sm B_\ep (x)} \lvert x-\xi \rvert^{-m}d\xi -\sigma \( S^{m-1}\) \log \frac{1}{\ep}  &(\al =0) 
\end{cases}
\]
whenever $\ep < \dist (x ,\Ome^c)$. In particular, for $\al <0$ and $x \in \c{\Ome}$, we have
\[
V_\Ome^{(\al )}(x) =- \int_{\Ome^c} \lvert x -\xi \rvert^{\al -m}d\xi .
\]
\end{prop}

Since this statement will play an important role in this paper, we review its proof (see also \cite[Lemma 2.4]{O3}).

\begin{proof}
We show the statement in the case of $\al <0$. Replacing the renormalization term, the proof in the case of $\al=0$ goes parallel.

Fix an arbitrary interior point $x$ of $\Ome$. Let $0<\de < \ep < \dist (x,\Ome^c)$. Since we have 
\[
\Ome \sm B_\ep (x) =(\Ome \sm B_\de(x)) \sm ( B_\ep (x) \sm B_\de (x)) ,
\] we get
\begin{align*}
\int_{\Ome \sm B_\ep (x)} \lvert x-\xi \rvert^{\al-m} d\xi 
&=\( \int_{\Ome \sm B_\de (x)} \lvert x-\xi \rvert^{\al -m} d\xi -\frac{\sigma \( S^{m-1}\)}{-\al} \de^\al \) \\
&\quad - \( \int_{B_\ep (x) \sm B_\de (x)} \lvert x- \xi \rvert^{\al-m} d\xi -\frac{\sigma \( S^{m-1}\)}{-\al} \de^\al \) .
\end{align*}
Since the left hand side is independent of $\de$, taking the limit $\de \to 0^+$, we obtain
\[
\int_{\Ome \sm B_\ep (x)} \lvert x-\xi \rvert^{\al-m} d\xi  
= V_\Ome^{(\al)} (x) + \frac{\sigma \( S^{m-1}\)}{-\al} \ep^\al ,
\]
which completes the proof.
\end{proof}

Proposition \ref{without_limit} and the definition \eqref{complementV} guarantee the continuity of $V_\Ome^{(\al)}$.

\begin{prop}[{\cite[Proposition 2.12]{O3}}]\label{cV}
Let $\Ome$ be a body in $\R^m$. For $\al \leq 0$, the restrictions of $V_\Ome^{(\al)}$ to the interior and the complement of $\Ome$ are continuous.
\end{prop}

Assuming the uniform boundary inner cone condition and using Proposition \ref{without_limit}, we can understand the behavior of the potential $V_\Ome^{(\al)}$ near the boundary of $\Ome$.

\begin{lem}[{\cite[Lemma 2.13]{O3}}]\label{boundaryV}
Let $\Ome$ be a body in $\R^m$, and $\al \leq 0$. 
\begin{enumerate}
\item[$(1)$] If the complement of $\Ome$ satisfies the uniform boundary inner cone condition, then the potential $V_\Ome^{(\al)}(x)$ diverges to $-\infty$ uniformly as $x \in \c{\Ome}$ approaches to any boundary point of $\Ome$.
\item[$(2)$] If the interior of $\Ome$ satisfies the uniform boundary inner cone condition, then the potential $V_\Ome^{(\al)}(x)$ diverges to $+\infty$ uniformly as $x \in \Ome^c$ approaches to any boundary point of $\Ome$.
\end{enumerate}
\end{lem} 

This Lemma will play an important role in section 3. Let us review its proof.

\begin{proof}
We show the first assertion. The proof of the second assertion goes parallel.

Suppose that the complement of $\Ome$ satisfies the uniform boundary inner cone condition for an open cone $C$ of vertex $0$, aperture angle $\kappa$ and height $\de$. Fix an arbitrary point $z$ on the boundary of $\Ome$. We can take an open cone $C(z)$ of vertex $z$ contained in the complement of $\Ome$ and congruent to $C$. 

We first show the statement in the case of $\al <0$. Let $\ep$ be a positive constant, and take a point $x \in \c{\Ome} \cap B_\ep (z)$. By Proposition \ref{without_limit}, we can estimate the potential $V_\Ome^{(\al)}$ as
\[
-V_\Ome^{(\al)}(x)
\geq \int_{C(z)} \lvert x-\xi \rvert^{\al-m}d\xi
\geq \frac{\sigma\( C \cap \de S^{m-1}\)}{\de^{m-1}} \int_0^\de \( \rho +\ep \)^{\al-m} \rho^{m-1} d\rho
\]
which diverges to $+\infty$ as $\ep$ tends to zero.

Next, we consider the case of $\al=0$. By proposition \ref{without_limit}, for any interior point $x$ of $\Ome$ and positive constant $\ep < \dist (x ,\Ome^c)$, we have
\begin{align*}
&\int_{\Ome \sm B_\ep (x)} \lvert x-\xi \rvert^{-m} d\xi - \sigma \( S^{m-1} \) \log \frac{1}{\ep} \\
&=\int_{\Ome \sm B_\ep(x)} \lvert x-\xi \rvert^{-m} d\xi -\int_{B_{\diam \Ome}(x) \sm B_\ep (x)} \lvert x- \xi \rvert^{-m} d\xi + \sigma \( S^{m-1} \) \log \diam \Ome \\
&=\sigma \( S^{m-1} \) \log \diam \Ome -\int_{B_{\diam \Ome}(x) \sm \Ome} \lvert x- \xi \rvert^{-m} d\xi .
\end{align*}
Thus, in the same argument as in the case of $\al <0$, we obtain the conclusion.
\end{proof}

Thanks to Proposition \ref{cV} and Lemma \ref{boundaryV}, for $\al \leq 0$, the restriction of  $V_\Ome^{(\al)}$ to the interior of $\Ome$ has a maximizer.

\begin{thm}[{\cite[Theorem 3.5]{O3}}]\label{existenceV}
Let $\Ome$ be a body in $\R^m$. Suppose that the complement of $\Ome$ satisfies the uniform boundary inner cone condition. For $\al \leq 0$, the restriction of $V_\Ome^{(\al)}$ to the interior of $\Ome$ has a maximizer.
\end{thm}

\begin{defi}[{\cite[Definition 3.1]{O3}}]
{\rm Let $\Ome$ be a body in $\R^m$. An interior point $c$ of $\Ome$ is called an {\it $r^{\al -m}$-center} of $\Ome$ if it gives the maximum value of the restriction of $V_\Ome^{(\al )}$ to the interior of $\Ome$. Let us denote the set of $r^{\al -m}$-centers by $\mathcal{V}_\Ome (\al )$, that is, 
\[
\mathcal{V}_\Ome (\al ) = \left\{ c \in \c{\Ome} \lvert V_\Ome^{(\al )} (c)=\max_{x \in \c{\Ome}} V_\Ome^{(\al)} (x) \right\} \right. .
\]
}
\end{defi}

\begin{rem}\label{Mos}
{\rm The name of a maximizer of $V_\Ome^{(\al)}$, $r^{\al-m}$-center, is originated in \cite{M1}. Moszy\'{n}ska defined a {\it radial center} of a star body $A$ as a maximizer of a function of the form
\[
\Phi_A (x) = \int_{S^{m-1}} \f \( \rho_{A -x} (v) \) d\sigma (v),\ 
x \in \Ker A := \left\{ \xi \in A \lvert \forall \eta \in A,\ \overline{\xi \eta} \subset A \right\} \right. ,
\]
where $\rho_{A-x}(v) = \max \{ \la \geq 0 \vert x +\la v \in A \}$ is the {\it radial function} of $A$ with respect to $x$, and $\overline{\xi \eta}$ denotes the line segment from $\xi$ to $\eta$. If $\f (r)=r^\al /\al$ ($0<\al <m$), then the function $\Phi_A$ coincides the Riesz potential $V_A^{(\al )}$.

Her motivation comes from the study on the {\it intersection body} of a star body. Intersection bodies were introduced by Lutwak in \cite{L} to given an affirmative answer to Busemann and Petty's problem \cite{BP}. The intersection body of a star body $A$ is defined by the radial function as $\rho_{IA} = \Vol_{m-1} ( A \cap v^\perp)$. Thus, the definition depends on the position of the origin. In \cite{M1}, Moszy\'{n}ska looked for an optimal position of the origin (see also \cite[Part III]{M2}). 

We refer to \cite{HMP} for the physical meaning of the study on centers of a body. The uniqueness of a radial center was discussed in \cite{H1, M1} but the investigation in \cite{H1} has an error, and it was pointed out in \cite{O3, Sak2}.
}
\end{rem}

Using the form of the potential $V_\Ome^{(\al)}$ in Remark \ref{Mos}, we can show the concavity of $V_\Ome^{(\al)}$ if $\Ome$ is convex.

\begin{thm}[{\cite[Theorem 3.12]{O3}}]\label{uniquenessV}
Let $\Ome$ be a convex body in $\R^m$. For $\al \leq 1$, the potential $V_\Ome^{(\al)}$ is strictly concave on $\Ome$. In particular, $\Ome$ has a unique $r^{\al -m}$-center.
\end{thm}
\subsection{Properties of the solid angle function {\boldmath $A_\Ome$}}
Let $\Ome$ be the closure of an open set in $\R^m$. We consider the solid angle of $\Ome$ at $(x,h ) \in \R^m \times (0,+\infty )$. From its definition mentioned in the introduction, we can show the following properties:
\begin{align}
\label{total_angle}
A_{\R^m} (x,h) &= \frac{\sigma \( S^m\)}{2},\ x \in \R^m ,\ h>0,\\
\label{local_angle}
\lim_{h\to 0^+} A_\Ome (x,h) &= \frac{\sigma \( S^m\)}{2} \chi_\Ome (x) ,\ x \in \R^m \sm \pd \Ome .
\end{align}

In \cite{Sak1, Sak2}, the author investigated properties of the solid angle function $A_\Ome$. Let us prepare some terminologies and properties of $A_\Ome$ from \cite{Sak1}.

Since the integrand of $A_\Ome$ is strictly decreasing with respect to $\lvert x- \xi \rvert$, for any point $p$ in the complement of the convex hull of $\Ome$, taking a point $p'$ on the boundary of the convex hull of $\Ome$ with $\vert p-p' \vert =\dist (p, (\conv \Ome )^c)$, we obtain $A_\Ome (p,h) < A_\Ome (p',h)$. Hence the continuity of $A_\Ome (\cdot ,h)$ imply the existence of a maximizer of $A_\Ome(\cdot ,h)$ if $\Ome$ is compact.

\begin{prop}[{\cite[Proposition 5.16]{Sak1}}]\label{existenceA}
Let $\Ome$ be a body in $\R^m$. For any $h>0$, the solid angle function $A_\Ome (\cdot ,h)$ has a maximizer, and all of them are contained in the convex hull of $\Ome$.
\end{prop}

\begin{defi}[{\cite[Definition 5.23]{Sak1}}]
{\rm Let $\Ome$ be a body in $\R^m$. A point $c$ is called an {\it illuminating center} of $\Ome$ of height $h$ if it gives the maximum value of $A_\Ome (\cdot ,h)$. Let us denote the set of illuminating centers by $\mathcal{A}_\Ome (h)$, that is,
\[
\mathcal{A}_\Ome (h) = \left\{ c \in \R^m \lvert A_\Ome (c,h) = \max_{x \in \R^m} A_\Ome (x,h) \right\} \right. .
\]
}
\end{defi}

The derivative of $A_\Ome (\cdot ,h)$ vanishes at a point $x$ if and only if the point $x$ satisfies the equation
\begin{equation}
x= \int_\Ome \( \lvert x-\xi \rvert^2 +h^2 \)^{-(m+3)/2} \xi d\xi \left/ \int_\Ome \( \lvert x-\xi \rvert^2 +h^2 \)^{-(m+3)/2} d\xi \right. ,
\end{equation}
which tells us the limiting point of an illuminating center of height $h$ as $h$ goes to infinity.

\begin{prop}[{\cite[Proposition 5.19]{Sak1}}]
Let $\Ome$ be a body in $\R^m$. The set of illuminating centers converges to the one-point set of the centroid of $\Ome$ as $h$ goes to infinity with respect to the Hausdorff distance.
\end{prop}
The small-height behavior of illuminating centers will be investigated in Theorem \ref{behaviorA1}.
\subsection{Properties of a potential with a radially symmetric kernel}
Let $\Ome$ be a body (the closure of a bounded open set) in $\R^m$. We consider a potential of the form
\begin{equation}
K_\Ome (x)=\int_\Ome k\( \lvert x-\xi \rvert \) d\xi,\ x \in \R^m.
\end{equation}
We understand that the kernel $k$ satisfies the condition $(C_\be^0)$ for a positive $\be$ if $k$ is continuous on the interval $(0,+\infty)$, and if
\begin{equation}
k(r)=
\begin{cases}
O\(r^{\be -m}\) &(\be <m),\\
O\( \log r\) &(\be =m),\\
O(1) &(\be >m) 
\end{cases}
\end{equation}
as $r$ tends to $0^+$.

In \cite{Sak1}, the author investigated properties of the potential $K_\Ome$. Let us prepare some terminologies and properties of $K_\Ome$ from \cite{Sak1}.

Let $\psi$ be a smooth function so that $\psi (r)=0$ if $0\leq r \leq 1$, $0 \leq \psi'(r) \leq 2$ if $1\leq r\leq 2$, and $\psi (r)=1$ if $2 \leq r$. If $k$ satisfies the condition $(C^0_\be)$ for some $\be >0$, then, for each positive $\ep$, the function 
\begin{equation}
\R^m \ni x \mapsto \int_\Ome k\( \lvert x-\xi \rvert \) \psi \( \frac{\lvert x- \xi \rvert}{\ep} \) d\xi \in \R
\end{equation}
is continuous and converges to $K_\Ome$ uniformly on $\R^m$ as $\ep$ tends to $0^+$. Thus, we obtain the continuity of $K_\Ome$.

\begin{prop}[{\cite[Proposition 2.3]{Sak1}}]\label{c0K}
Let $\Ome$ be a body in $\R^m$. If $k$ satisfies the condition $(C_\be^0)$ for some $\be >0$, then the potential $K_\Ome$ is continuous on $\R^m$.
\end{prop}

In the same manner as Proposition \ref{existenceA}, we can show the existence of a maximizer of $K_\Ome$.

\begin{prop}[{\cite[Proposition 3.2]{Sak1}}]
Let $\Ome$ be a body in $\R^m$. If $k$ is strictly decreasing and satisfies the condition $(C_\be^0)$ for some $\be >0$, then the potential $K_\Ome$ has a maximizer, and all of them are contained in the convex hull of $\Ome$.
\end{prop}

\begin{defi}[{\cite[Definition 3.3]{Sak1}}]
{\rm Let $\Ome$ be a body in $\R^m$. A point $c$ is called a {\it $k$-center} of $\Ome$ if it gives the maximum value of $K_\Ome$. We denote the set of $k$-centers of $\Ome$ by $\mathcal{K}_\Ome$, that is,
\[
\mathcal{K}_\Ome = \left\{ c \in \R^m \lvert K_\Ome (c) =\max_{x \in \R^m} K_\Ome (x) \right\} \right. .
\]
For the potential $K_\Ome(x,t)$ defined in \eqref{K}, we call a maximizer of $K_\Ome (\cdot ,t)$ a {\it $k$-center at time $t$}.
}
\end{defi}

\begin{prop}\label{concavity}
Let $\Ome$ be a convex body in $\R^m$. Let $\Ome'$ be a convex body contained in the interior of $\Ome$. Put
\[
d(\Ome ,\Ome') = \inf \left\{ \lvert z-w \rvert \lvert z \in \pd \Ome ,\ w \in \Ome' \right\} \right. ,\ 
D(\Ome ,\Ome') = \sup \left\{ \lvert z-w \rvert \lvert z \in \pd \Ome ,\ w \in \Ome' \right\} \right. .
\]
Let $k$ be positive and satisfy the condition $(C^0_\be)$ for some $\be >0$. If $k(r)r^{m-1}$ is decreasing for $r \in [d(\Ome ,\Ome'), D(\Ome ,\Ome') ]$, then $K_\Ome$ is strictly concave in $\Ome'$.
\end{prop}

\begin{proof}
We take distinct two points $x$ and $y$ from $\Ome'$. Using the polar coordinate, we have
\begin{align*}
&\quad 2K_\Ome \( \frac{x+y}{2}\) - \( K_\Ome (x) +K_\Ome (y) \) \\
&= \int_{S^{m-1}} \( \( 2\int_0^{\rho_{\Ome -(x+y)/2}(v)} -\int_0^{\rho_{\Ome -x}(v)}-\int_0^{\rho_{\Ome -y}(v)}\) k (r)r^{m-1}dr\)d\sigma (v) \\
&> \int_{S^{m-1}} \( \( 2\int_0^{\( \rho_{\Ome -x}(v)+ \rho_{\Ome -y}(v)\)/2} -\int_0^{\rho_{\Ome -x}(v)}-\int_0^{\rho_{\Ome -y}(v)}\) k (r)r^{m-1}dr\)d\sigma (v) \\
&= \int_{S^{m-1}} \( \( \int_{\min \left\{ \rho_{\Ome-x}(v), \ \rho_{\Ome-y}(v) \right\}}^{\( \rho_{\Ome -x}(v) + \rho_{\Ome -y}(v)\)/2} - \int_{\( \rho_{\Ome -x}(v) + \rho_{\Ome -y}(v)\)/2}^{\max \left\{ \rho_{\Ome-x}(v), \ \rho_{\Ome-y}(v) \right\}}  \) k (r)r^{m-1}dr\)d\sigma (v)\\
&\geq 0.
\end{align*}
Here, the first and second inequalities follow from the convexity of $\Ome$ and the decreasing behavior of $k(r)r^{m-1}$, respectively.
\end{proof}
\subsection{The minimal unfolded region of a body}
Let $\Ome$ be a body (the closure of a bounded open set) in $\R^m$. Using the radial symmetry of the kernels of the potentials mentioned in the previous subsections, we can restrict the region containing those centers. We introduce the restricted region and its properties from \cite{BM, O3, Sak1} (see also \cite{BMS, O4, Sak2}).

\begin{defi}[{\cite[Definition 3.3]{O3}}]\label{uf}
{\rm Let $v$ be a direction in the unit sphere $S^{m-1}$, and $b$ a real parameter. Let $\Refl_{v,b}$ be the reflection of $\R^m$ in the hyperplane $\{ z \in \R^m \vert z \cdot v =b \}$. Put 
\[
\Ome_{v,b}^+ = \Ome \cap \left. \left\{ z \in \R^m \rvert z \cdot v \geq b \right\} ,\ 
l(v)= \min \left\{ a\in \R \lvert \forall b \geq a,\ \Refl_{v,b}\( \Ome_{v,b}^+ \) \subset \Ome  \right\} \right. 
\] 
(see Figure \ref{uffig}). Define the {\it minimal unfolded region} of $\Ome$ by
\[
Uf(\Ome ) = \bigcap_{v \in S^{m-1}} \left. \left\{ z \in \R^m \rvert z \cdot v \leq l(v) \right\} .
\]
\begin{figure}
\centering
\scalebox{0.39}{\includegraphics[clip]{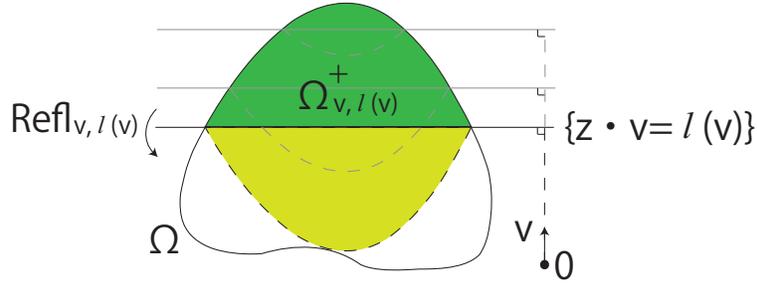}}
\caption{Folding a body $\Ome$ in the manner in Definition \ref{uf}}
\label{uffig}
\end{figure}
}
\end{defi} 

\begin{ex}[{\cite[Lemma 5]{BM}, \cite[Example 3.4]{O3}}]\label{uf_triangle}
{\rm 
\begin{enumerate}
\item[(1)] The minimal unfolded region of the disjoint union of three discs is surrounded by the lines through two centers of discs (see Figure \ref{three_discs}).
\item[(2)] The minimal unfolded region of an acute triangle is surrounded by the mid-perpendiculars of edges and the bisectors of angles (see Figure \ref{acute}).
\item[(3)] The minimal unfolded region of an obtuse triangle is surrounded by the largest edge, its midperpendicular and the bisectors of angles (see Figure\ref{obtuse}). We remark that the minimal unfolded region of $\Ome$ is not always contained in the interior of the convex hull of $\Ome$ even if $\Ome$ is convex.
\end{enumerate}
\begin{figure}[htbp]
\begin{center}
\begin{minipage}[h]{0.3\linewidth}
\centering
\scalebox{0.3}{\includegraphics[clip]{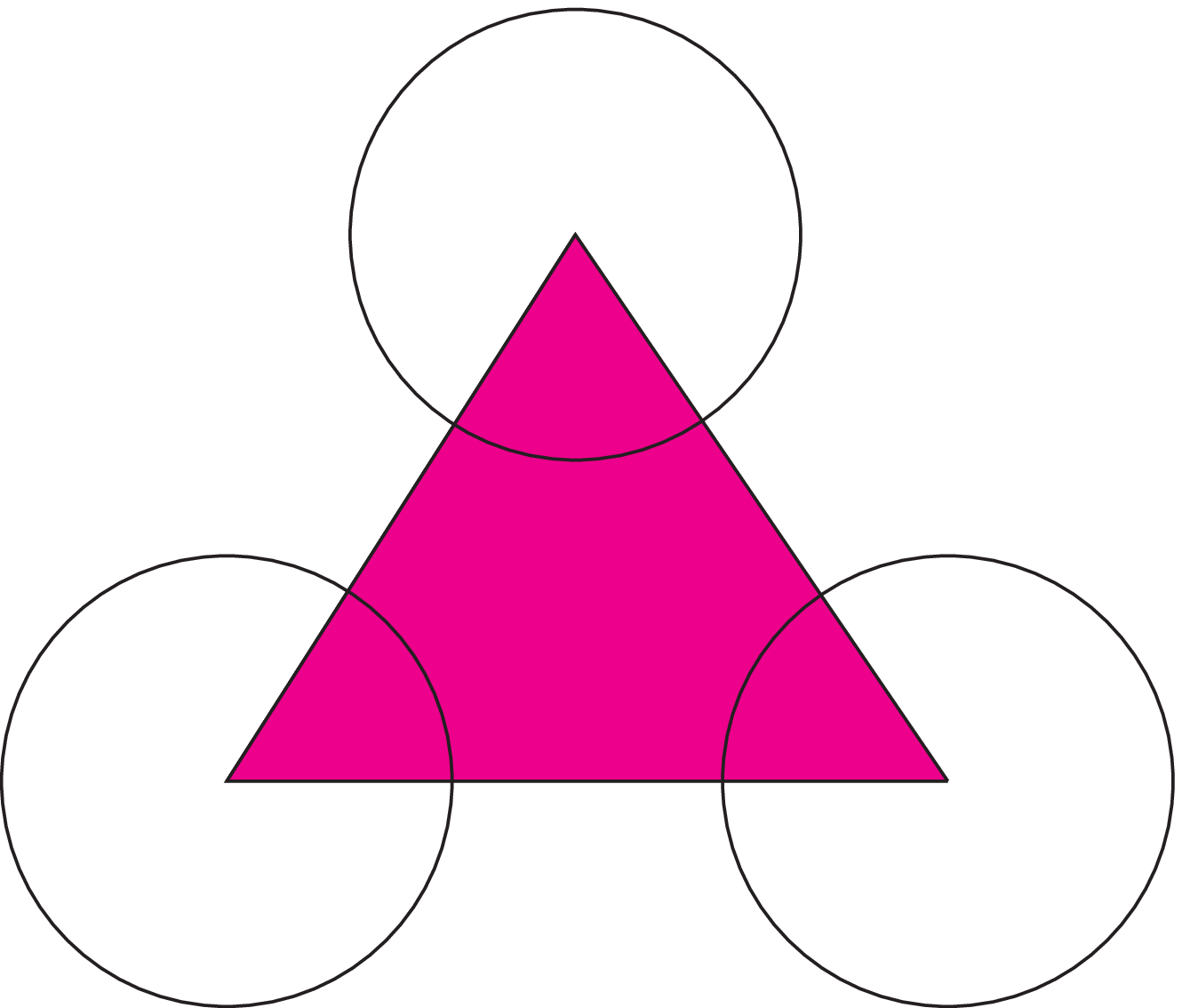}}
\caption{The minimal unfolded region of the disjoint union of three discs}
\label{three_discs}
\end{minipage}
\hspace{0.01\linewidth}
\begin{minipage}[h]{0.3\linewidth}
\centering
\scalebox{0.25}{\includegraphics[clip]{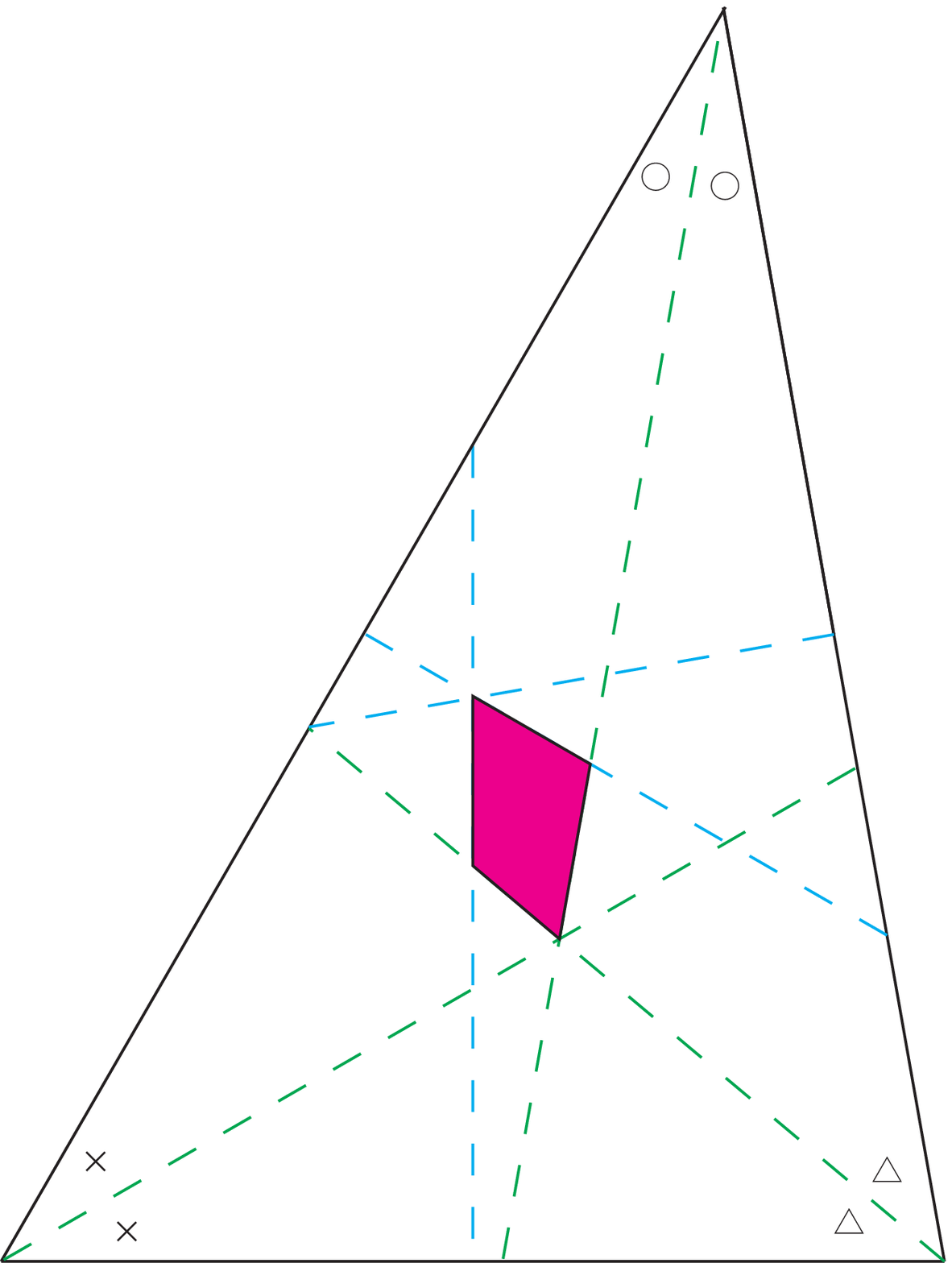}}
\caption{The minimal unfolded region of an acute triangle}
\label{acute}
\end{minipage}
\hspace{0.01\linewidth}
\begin{minipage}[h]{0.3\linewidth}
\centering
\scalebox{0.3}{\includegraphics[clip]{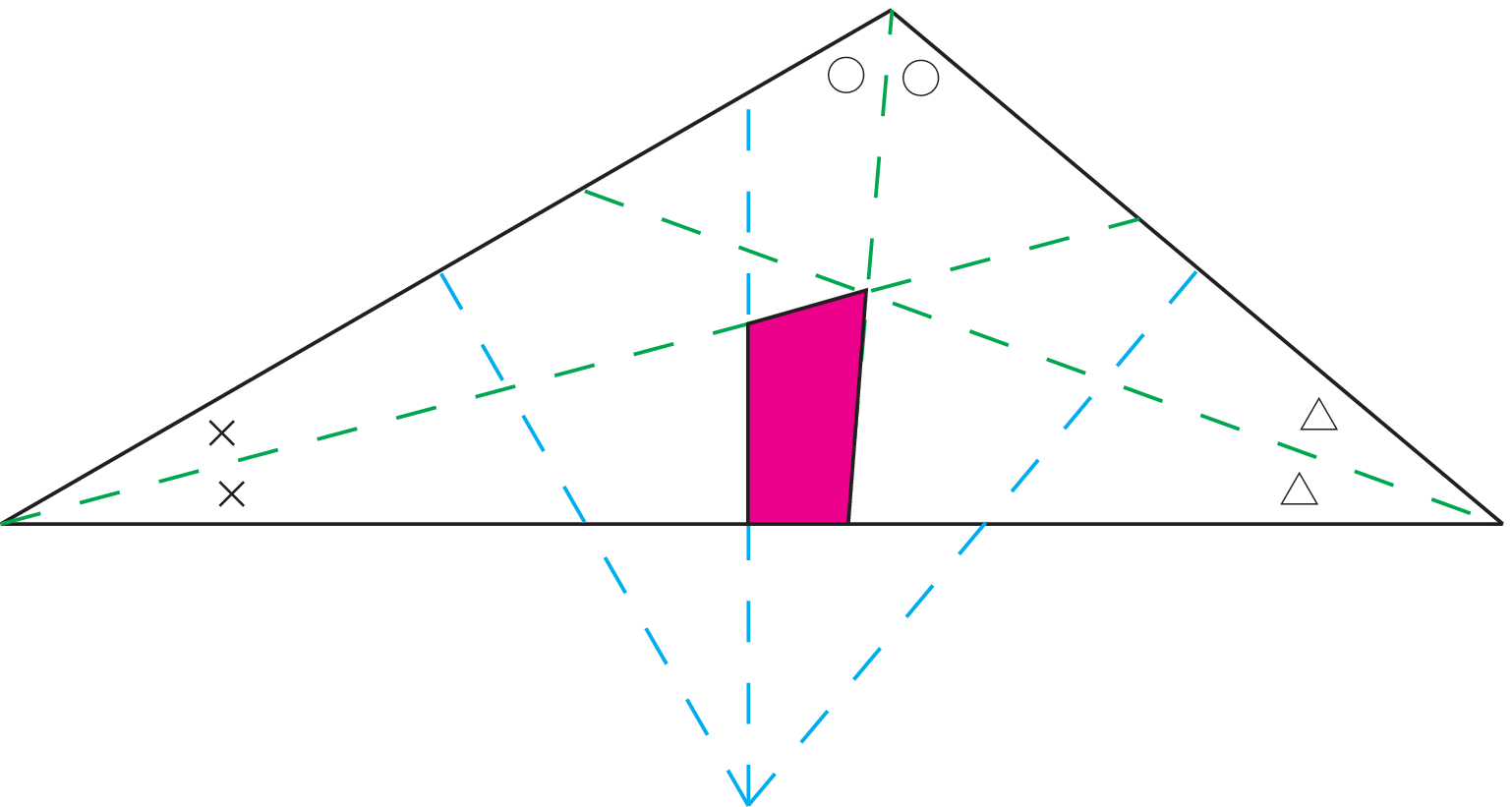}}
\caption{The minimal unfolded region of an obtuse triangle}
\label{obtuse}
\end{minipage}
\end{center}
\end{figure}
}
\end{ex}

\begin{rem}[{\cite[Proposition 1]{BM}, \cite[p. 381]{O3}}]
{\rm
\begin{enumerate}
\item[(1)] The centroid (center of mass) of $\Ome$ is contained in $Uf(\Ome )$. Hence $Uf(\Ome )$ is not empty.
\item[(2)] $Uf(\Ome )$ is compact and convex.
\item[(3)] $Uf(\Ome )$ is contained in the convex hull of $\Ome$.
\end{enumerate}
}
\end{rem}

Using the {\it moving plane method} (\cite{GNN, Ser}), we can restrict the location of the centers prepared in the previous subsections.

\begin{prop}\label{ufV}
Let $\Ome$ be a body in $\R^m$. For $\al \leq 0$, any $r^{\al -m}$-center of $\Ome$ belongs to the minimal unfolded region of $\Ome$.
\end{prop}

\begin{proof}
Let $x$ be an interior point of $\Ome$ in the complement of the minimal unfolded region of $\Ome$. We show that the point $x$ is not an $r^{\al-m}$-center of $\Ome$.

We can choose a direction $v \in S^{m-1}$ with $l(v) <x\cdot v$. Let $b =(l(v)+x\cdot v)/2$. Then, the region $\Refl_{v,b} (\Ome^+_{v,b})$ is contained in $\Ome$, and $\Ome \sm (\Ome^+_{v,b} \cup  \Refl_{v,b} (\Ome^+_{v,b}))$ has an interior point. 

Let $x'=\Refl_{v,b}(x)$. We choose a small enough $\ep >0$ so that the ball $B_\ep (x)$ is contained in the interior of $\Ome$. Then, we have the following properties:
\begin{align*}
\int_{\Ome^+_{v,b} \sm B_\ep (x)} \lvert x-\xi \rvert^{\al-m} d\xi
&=\int_{\Refl_{v,b}\( \Ome^+_{v,b} \) \sm B_\ep \(x'\)} \lvert x'-\xi \rvert^{\al-m} d\xi ,\\
\int_{\Refl_{v,b}\( \Ome^+_{v,b} \)} \lvert x-\xi \rvert^{\al-m} d\xi
&=\int_{\Ome^+_{v,b}}\lvert x'-\xi \rvert^{\al-m} d\xi .
\end{align*}
Furthermore, for any point $\xi \in \Ome \sm (\Ome^+_{v,b} \cup  \Refl_{v,b} (\Ome^+_{v,b}))$, we have $\lvert x-\xi \rvert^{\al-m} < \lvert x'-\xi \rvert^{\al-m}$. Hence, by Proposition \ref{without_limit}, we obtain
\begin{align*}
V_\Ome^{(\al)} (x)-V_\Ome^{(\al)}\(x'\)
&=\int_{\Ome \sm B_\ep(x)} \lvert x-\xi \rvert^{\al-m} d\xi -\int_{\Ome \sm B_\ep \( x' \)} \lvert x'-\xi \rvert^{\al-m} d\xi \\
&= \( \int_{\Ome^+_{v,b} \sm B_\ep (x)} +\int_{\Refl_{v,b} \( \Ome^+_{v,b}\)} + \int_{\Ome \sm \(\Ome^+_{v,b} \cup  \Refl_{v,b} \( \Ome^+_{v,b}\)\)} \) \lvert x -\xi \rvert^{\al-m} d\xi \\
&\quad - \( \int_{\Refl_{v,b}\( \Ome^+_{v,b}\) \sm B_\ep \( x'\)} +\int_{\Ome^+_{v,b}} + \int_{\Ome \sm \(\Ome^+_{v,b} \cup  \Refl_{v,b} \( \Ome^+_{v,b}\)\)} \) \lvert x' -\xi \rvert^{\al-m} d\xi \\
&<0,
\end{align*}
which completes the proof.
\end{proof}

\begin{rem}
{\rm In \cite[Theorem 3.5]{O3}, O'Hara asserted the same statement as Proposition \ref{ufV} when $\Ome$ has a piecewise $C^1$ boundary. But, in this paper, we does not assume the smoothness of a body $\Ome$.}
\end{rem}

In the same manner as in Proposition \ref{ufV}, we can restrict the location of $k$-centers of $\Ome$ into the minimal unfolded region of $\Ome$.

\begin{prop}[{\cite[Proposition 4.9]{Sak1}}]\label{ufK}
Let $\Ome$ be a body in $\R^m$. If $k$ is strictly decreasing and satisfies the condition $(C_\be^0)$ for some $\be >0$, then any $k$-center of $\Ome$ belongs to the minimal unfolded region of $\Ome$.
\end{prop}

\begin{rem}
{\rm We refer to \cite{H2} for the location of $r^{\al-m}$-centers. Herburt showed that the (unique) $r^{1-m}$-center of a smooth convex body $A$ belongs to the interior of $A$.

For $\al \leq 0$, we discuss the location of $r^{\al -m}$-centers in Theorem \ref{estimationV} and Corollary \ref{centerV}. For $\al >1$, any $r^{\al -m}$-center of a body $\Ome$ belongs to the intersection $Uf(\Ome ) \cap ( \conv \Ome )^\circ$. This statement follows from the fact that, for a boundary point $x$ of $\conv \Ome$ and the unit outer normal field $n$ of $\conv \Ome$, the derivative
\[
\frac{\pd V_\Ome^{(\al)}}{\pd n(x)} (x) = (\al -m) \int_\Ome r^{\al -m-2} ( x-y ) \cdot n(x) dy
\]
does not vanish.

Herburt's theorem does not follow from the same argument as in the case of $\al >1$. This is because the potential $V_\Ome^{(1)}$ is not differentiable at any boundary point of $\Ome$. Also, the minimal unfolded region of $\Ome$ touches the boundary of $\Ome$ in general (see Example \ref{uf_triangle}).

Hence, for $0<\al <1$, the location of $r^{\al -m}$-centers is unknown. 
}
\end{rem}
\section{Estimation of an {\boldmath $r^{\al-m}$}-potential}
Let $\Ome$ be a body (the closure of a bonded open set) in $\R^m$. By Proposition \ref{ufV}, the set of $r^{\al -m}$-centers ($\al \leq 0$) of $\Ome$ is contained in the minimal unfolded region of $\Ome$. But, by Lemma \ref{boundaryV}, it is expected that any $r^{\al -m}$-center does not exist ``near'' the boundary of $\Ome$. For example, when $\Ome$ is an obtuse triangle in $\R^2$, the minimal unfolded region of $\Ome$ touches the boundary of $\Ome$, but it is expected that any $r^{\al -2}$-center belongs to a smaller closed region contained in the interior of the minimal unfolded region of $\Ome$. Let us show that the expectation is true when the complement of $\Ome$ satisfies the uniform boundary inner cone condition.

Let $C(x)=C(x;\kappa ,\de)$ be an open cone of vertex $x$, axis direction $e_1$, aperture angle $0< \kappa \leq \pi$ and height $0< \de \leq +\infty$, that is,
\begin{align}
\nonumber
C(x)
&=C(x;\kappa ,\de) \\
\label{cone0}
&= \left. \left\{ \rho\(
\begin{array}{c}
\cos \phi_1\\
\sin \phi_1 \cos \phi_2\\
\vdots \\
\sin \phi_1 \cdots \sin \phi_{m-2} \cos \phi_{m-1} \\
\sin \phi_1 \cdots \sin \phi_{m-2} \sin \phi_{m-1} 
\end{array}
\) +x \rvert 0 < \rho < \de ,\ \phi \in \( -\frac{\kappa}{2} ,\frac{\kappa}{2} \) \times [ 0,\pi ]^{m-2} \right\}  ,
\end{align}
where $\phi = (\phi_1 ,\ldots , \phi_{m-1})$.
Let $\Rot_{1m}(\theta)$ denote the rotation in the plane $\Span \langle e_1 ,e_m\rangle$ of  angle $\theta$, that is, 
\begin{equation}\label{rot}
\Rot_{1m}(\theta ) = \( \begin{array}{ccccc}
\cos \theta & &       & &-\sin \theta \\
            &1&       & &             \\
            & &\ddots & &       \\
            & &       &1&       \\
\sin \theta & &       & &\cos \theta
\end{array}\) .
\end{equation}
Let
\begin{equation}\label{cone1}
C_\theta (x) = C_\theta (x; \kappa ,\de ) =\Rot_{1m} (\theta ) C(0; \kappa ,\de) +x .
\end{equation}

\begin{lem}\label{inf}
Let $\al \leq 0$, $0< \kappa \leq \pi$, $0< \de \leq +\infty$, and $0<R <D$. If $\de \leq D-R$ or $\de \geq \sqrt{D^2 -R^2}$, then we have
\[
\min_{0 \leq \theta \leq (\pi -\kappa )/2} \int_{C_\theta \( Re_1 \) \cap B_D (0)} \lvert \xi \rvert^{\al -m} d\xi 
= \int_{C \(Re_1 \) \cap B_D (0)} \lvert \xi \rvert^{\al -m} d\xi ,
\]
where $C_\theta (Re_1 ) =C_\theta (Re_1 ;\kappa ,\de)$ is the cone defined in \eqref{cone1}.
\end{lem}

\begin{proof}
We take a point $\xi (\theta ,\phi )$ from $C_\theta (Re_1)$ as
\[
\xi(\theta, \phi )
=\rho  \( \begin{array}{ccccc}
\cos \theta & &       & &-\sin \theta \\
            &1&       & &             \\
            & &\ddots & &       \\
            & &       &1&       \\
\sin \theta & &       & &\cos \theta
\end{array}\) 
\(
\begin{array}{c}
\cos \phi_1\\
\sin \phi_1 \cos \phi_2\\
\vdots \\
\sin \phi_1 \cdots \sin \phi_{m-2} \cos \phi_{m-1} \\
\sin \phi_1 \cdots \sin \phi_{m-2} \sin \phi_{m-1} 
\end{array}
\)+Re_1.
\]
We remark
\[
\lvert \xi (\theta ,\phi ) \rvert = \sqrt{\rho^2 +R^2 +2\rho R \( \cos \theta \cos \phi_1 -\sin \theta \sin \phi_1 \cdots \sin \phi_{m-2} \sin \phi_{m-1} \)} .
\]
In order to estimate the contribution of a point in the intersection $C_\theta (Re_1) \cap B_D(0)$ to the integral, let us show the non-negativity of the difference
\[
\De (\theta ,\phi )
:= \lvert \xi (0 ,\phi) \rvert - \lvert \xi (\theta ,\phi) \rvert -\lvert \lvert \xi \( 0 , \bar{\phi} \) \rvert -  \lvert \xi \( \theta ,\bar{\phi} \) \rvert \rvert 
\] 
for any $0\leq \theta \leq (\pi -\kappa) /2$ and $\phi \in [0,\kappa /2) \times [0,\pi]^{m-2}$ (see Figure \ref{contribution}), where $\bar{\phi} = ( \phi_1 , \ldots ,  \phi_{m-2},-\phi_{m-1} )$.
\begin{figure}[hbtp]
\centering
\scalebox{0.8}{\includegraphics[clip]{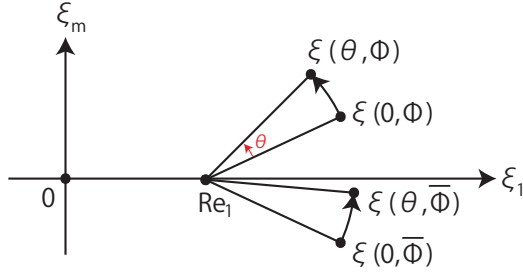}}
\caption{The location of $\xi (\theta ,\phi )$ and the difference $\De (\theta ,\phi)$.}
\label{contribution}
\end{figure}

If $\lvert \xi (0,\bar{\phi}) \rvert \geq  \lvert \xi ( \theta ,\bar{\phi} ) \rvert$, then we have
\[
\De (\theta ,\phi )=\lvert \xi \( \theta ,\bar{\phi}\) \rvert - \lvert \xi (\theta ,\phi) \rvert \geq 0.
\]
Let us consider the case of $\lvert \xi (0,\bar{\phi}) \rvert \leq  \lvert \xi (\theta ,\bar{\phi}) \rvert$. Then we have
\[
\De (\theta ,\phi )=2 \lvert \xi (0,\phi) \rvert - \lvert \xi (\theta ,\phi) \rvert -\lvert \xi \( \theta ,\bar{\phi}\) \rvert .
\]
It is sufficient to show the non-negativity of the difference
\[
4\lvert \xi (0,\phi) \rvert^2-\( \lvert \xi (\theta ,\phi) \rvert +\lvert \xi \( \theta ,\bar{\phi} \) \rvert \)^2 .
\]
Since we have
\[
2\lvert \xi (0,\phi) \rvert^2-\lvert \xi (\theta ,\phi) \rvert^2 -\lvert \xi \( \theta ,\bar{\phi} \) \rvert^2
=4\rho R \cos \phi_1 \( 1-\cos \theta\) \geq 0 ,
\]
we get
\begin{align*}
4\lvert \xi (0,\phi) \rvert^2-\( \lvert \xi (\theta ,\phi) \rvert +\lvert \xi \( \theta ,\bar{\phi} \) \rvert  \)^2 
&\geq 2\lvert \xi (0,\phi) \rvert^2-2\lvert \xi (\theta ,\phi) \rvert  \lvert \xi \( \theta ,\bar{\phi} \) \rvert \\
&\geq \lvert \xi (\theta ,\phi) \rvert^2 + \lvert \xi \( \theta ,\bar{\phi} \) \rvert^2 - 2\lvert \xi (\theta ,\phi) \rvert \lvert \xi \( \theta ,\bar{\phi} \) \rvert \\
&\geq 0 .
\end{align*}

In order to complete the proof, we prepare the following notation:
\begin{align*}
U_\theta &= \Rot_{1m}(\theta ) \( C\(Re_1\) \cap B_D(0) \cap \left\{ \xi_m \geq 0 \right\} -Re_1 \) +Re_1 ,\\
L_\theta &= \Rot_{1m}(\theta ) \( C\(Re_1\) \cap B_D(0) \cap \left\{ \xi_m \leq 0 \right\} -Re_1 \) +Re_1 .
\end{align*}
The non-negativity of the difference $\De (\theta ,\phi)$ implies
\[
\( \int_{L_0} -\int_{L_\theta} \) \lvert \xi \rvert^{\al -m} d\xi
\leq \( \int_{U_\theta} -\int_{U_0} \) \lvert \xi \rvert^{\al -m} d\xi ,
\]
and hence, we get
\[
\int_{C \( Re_1 \) \cap B_D (0)} \lvert \xi \rvert^{\al -m} d\xi 
=\int_{U_0 \cup L_0} \lvert \xi \rvert^{\al -m} d\xi 
\leq \int_{U_\theta \cup L_\theta} \lvert \xi \rvert^{\al -m} d\xi .
\]

If $\de \leq D-R$, then $C_\theta (Re_1 ) = C_\theta (Re_1 ) \cap B_D (0) = U_\theta \cup L_\theta$, that is, the proof is completed in this case. Let us consider the case of $\de \geq \sqrt{D^2 -R^2}$. Using the non-negativity of $\De (\theta, \phi)$, we can show
\[
\Vol \( \( U_\theta \cup L_\theta \) \sm \( C_\theta \( Re_1\) \cap B_D(0) \)\)
\leq \Vol \( \( C_\theta \( Re_1\) \cap B_D(0) \) \sm \( U_\theta \cup L_\theta \) \) 
\]
(see Figure \ref{cones}). Hence we obtain
\[
\int_{U_\theta \cup L_\theta} \lvert \xi \rvert^{\al-m}d\xi
\leq \int_{C_\theta \( Re_1\) \cap B_D(0)} \lvert \xi \rvert^{\al-m}d\xi ,
\]
which completes the proof .
\begin{figure}[hbtp]
\centering
\scalebox{0.5}{\includegraphics[clip]{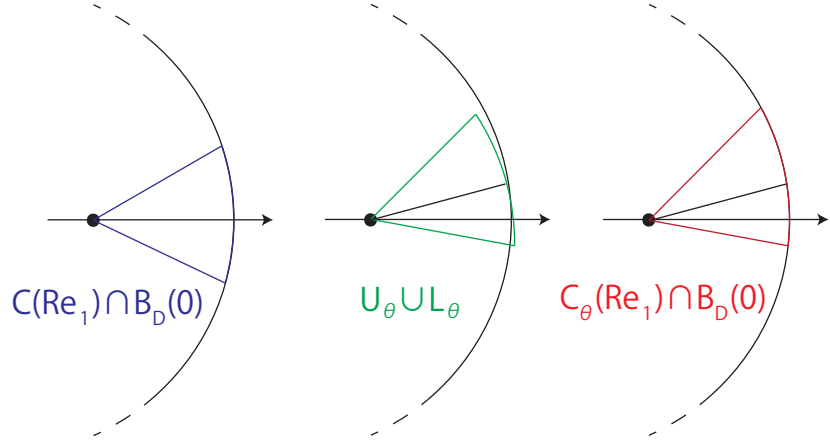}}
\caption{The estimation of the integrals}
\label{cones}
\end{figure}
\end{proof}

\begin{lem}\label{E} 
Let $\al \leq 0$, $0 < \kappa \leq \pi$, $0< \de \leq +\infty$, $D>0$, and $0<R_0 <D$. Define the function
\[
E (R) = E\( R; \al ,\kappa ,\de , D, R_0 \) =
\min_{0 \leq \theta \leq (\pi -\kappa)/2} \( \int_{C_\theta \(Re_1 \) \cap B_D(0)}  -\int_{B_D (0) \sm B_{R_0}(0)} \) \lvert \xi \rvert^{\al -m} d\xi ,\ R>0,
\]
where $C_\theta (Re_1) = C_\theta (Re_1 ;\kappa ,\de)$ is the cone defined in \eqref{cone1}.
\begin{enumerate}[$(1)$]
\item The function $E$ is strictly decreasing.
\item There exists a unique positive constant $\tilde{R} = \tilde{R} (\al ,\kappa ,\de ,D, R_0)$ such that $E(R)>0$ if $R < \tilde{R}$, and that $E (R)<0$ if $R> \tilde{R}$. In particular, $\tilde{R}$ is the unique zero point of $E$. 
\item The unique zero point $\tilde{R}$ is less than $R_0$.
\end{enumerate}
\end{lem}

\begin{proof}
(1) Let $0< R_1 <R_2$. We denote by $\theta_j$ an angle giving the minimum value in the definition of $E (R_j)$. The strictly decreasing behavior of the function $r \mapsto r^{\al -m}$ implies
\begin{align*}
E \( R_1 \) 
&= \( \int_{C_{\theta_1} \(R_1e_1\) \cap B_D(0)}  -\int_{B_D (0) \sm B_{R_0}(0)} \) \lvert \xi \rvert^{\al -m} d\xi \\
&> \( \int_{C_{\theta_1} \( R_2 e_1\) \cap B_D(0)}  -\int_{B_D (0) \sm B_{R_0}(0)} \) \lvert \xi \rvert^{\al -m} d\xi \\
&\geq \( \int_{C_{\theta_2} \( R_2 e_1\) \cap B_D(0)}  -\int_{B_D (0) \sm B_{R_0}(0)} \) \lvert \xi \rvert^{\al -m} d\xi \\
&= E \( R_2 \) .
\end{align*}

(2) First, we remark that $E (R)$ is negative for $R \geq R_0$. This is because, for any $0 \leq \theta \leq (\pi -\kappa )/2$ and $R \geq R_0$, $C_\theta \( Re_1 \) \cap B_D (0)$ is contained in the annulus $B_D (0) \sm B_{R_0} (0)$. 

Next, we show that $E(R)$ diverges to $+\infty$ as $R \to 0^+$. We take a small enough $\ep >0$ so that $\ep \de <D-R_0$. Then, for any $0 \leq \theta \leq (\pi -\kappa )/2$ and $R \leq R_0$, the small cone $\ep C_\theta (Re_1 )$ is contained in the ball $B_D(0)$. From Lemma \ref{inf}, we have
\[
\min_{0 \leq \theta \leq (\pi -\kappa )/2} \int_{\ep C_\theta \( Re_1 \) \cap B_D (0)} \lvert \xi \rvert^{\al -m} d\xi 
= \int_{\ep C \(Re_1 \) \cap B_D (0)} \lvert \xi \rvert^{\al -m} d\xi .
\]
Therefore, Lemma \ref{boundaryV} implies 
\[
E(R) 
\geq \( \int_{\ep C \(Re_1 \) \cap B_D (0)} -\int_{B_D (0) \sm B_{R_0} (0)} \) \lvert \xi \rvert^{\al -m} d\xi \to +\infty 
\]
as $R\to 0^+$.

Hence the continuity of $E$ implies the existence and uniqueness of  a zero point of $E$.

(3) The statement was shown in the proof of (2) as $E(R)$ is negative for $R \geq R_0$.
\end{proof}

\begin{thm}\label{estimationV} 
Let $\al \leq 0$. Let $X$ and $Y$ be bodies in $\R^m$. Suppose that the complement of $Y$ satisfies the uniform boundary inner cone condition of aperture angle $\kappa$ and height $\de$. Let $R_0 >0$, and $\tilde{R} = \tilde{R} ( \al ,\kappa , \de ,\diam Y, R_0 )$ be as in Lemma \ref{E}. For any points $x \in X$ with $\dist (x, X^c) \geq R_0$ and $y \in  Y$ with $\dist (y,Y^c) \leq \tilde{R}$, we have $V_Y^{(\al)} (y) < V_X^{(\al)}(x)$.
\end{thm}

\begin{proof}
If $R_0$ is greater than half of the diameter of $Y$, then the statement obviously hods. Let us consider the case where $R_0$ is not greater than half of the diameter of $Y$. 

Fix an interior point $y$ of $Y$. Let $y'$ be a boundary point of $Y$ with $\vert y -y' \vert = \dist (y,Y^c )$. From the uniform boundary inner cone condition of the complement of $Y$, there is a direction $v(y)$ such that we can take an open cone of vertex $y'$, axis direction $v(y)$, aperture angle $\kappa$ and height $\de$. Let $\theta (y)$ be the angle between $(y' -y)/\vert y'- y\vert$ and $v(y)$. By radial symmetry of the kernel of $V_Y^{(\al )}$, we get
\[
V_Y^{(\al)}(y) 
< V_{B_{\diam Y}(0) \sm C_{\theta (y)} \(\dist \( y,Y^c \)e_1\)}^{(\al)} (0) ,
\]
where $C_{\theta (y)} \(\dist \( y,Y^c \)e_1\) = C_{\theta (y)} \(\dist \( y,Y^c \)e_1;\kappa ,\de \)$ is the cone defined in \eqref{cone1}. Also, we have
\[
V_X^{(\al )} (x) \geq V_{B_{R_0}(0)}^{( \al)} (0) .
\]
Hence, for any points $x \in X$ with $\dist (x, X^c) \geq R_0$ and $y \in  Y$ with $\dist (y,Y^c) \leq \tilde{R}$, we get 
\begin{align*}
V_X^{(\al )}(x)-V_Y^{(\al)}(y)
&> V_{B_{R_0}(0)}^{(\al)}(0) -V_{B_{\diam Y}(0)\sm C_{\theta (y)} \(\dist \( y,Y^c \) e_1\)}^{(\al)} (0) \\
&= \( \int_{C_{\theta(y)} \( \dist \( y,Y^c \) e_1\)\cap B_{\diam Y}(0)} - \int_{B_{\diam Y}(0) \sm B_{R_0}(0)} \) \lvert \xi \rvert^{\al -m}d\xi \\
&\geq E \( \dist \( y,Y^c \) \) \\
&\geq 0,
\end{align*}
where $E= E( \cdot ; \al ,\kappa ,\de ,\diam Y, R_0)$ is defined in Lemma \ref{E}.
\end{proof}

\begin{cor}\label{centerV} 
Let $\al \leq 0$. Let $\Ome$ be a body in $\R^m$ whose complement satisfies the uniform boundary inner cone condition of aperture angle $\kappa$ and height $\de$. Any $r^{\al -m}$-center of $\Ome$ belongs to the intersection $( \Ome \sim \tilde{R}B^m ) \cap Uf(\Ome )$, where $\tilde{R} = \tilde{R} ( \al ,\kappa, \de , \diam \Ome ,R_\infty (\Ome ))$ is given in Lemma \ref{E}.
\end{cor}

\begin{ex}\label{dumbbellV}
{\rm Let $\al \leq 0$. For $0\leq \ep \leq 1$, let 
\[
\Ome_\ep = \( [-3,-1] \times B^{m-1} \) \cup \( [-1,1]  \times \ep B^{m-1} \)  \cup \( [1,3] \times B^{m-1} \).
\]
We take an open cone $C$ of aperture angle $\kappa$ and height $\de$ such that the complement of $\Ome_0$ satisfies the uniform interior cone condition for $C$. Then, for any $0<\ep \leq 1$, the complement of the body $\Ome_\ep$ satisfies the uniform interior cone condition for $C$. We remark $\diam \Ome_\ep = 2\sqrt{10}$ and $R_\infty (\Ome_\ep)=1$ for any $0\leq \ep \leq 1$. Let $\tilde{R}= \tilde{R} ( \al ,\kappa ,\de ,3\sqrt{5} ,1)$ be as in Lemma \ref{E}, and fix an $0< \ep < \tilde{R}$. 

Since $Uf(\Ome ) = [-2, 2] \times \{ 0\}^{m-1}$ and $\tilde{R} <1$, Corollary \ref{centerV} implies that any $r^{\al -m}$-center of the body $\Ome_\ep$ belongs to the disjoint union of the intervals $( [ -2 ,-1 -\sqrt{\tilde{R}^2 -\ep^2} ] \cup [ 1+ \sqrt{\tilde{R}^2-\ep^2} ,2 ] ) \times \{ 0 \}^{m-1}$.
Radial symmetry of the kernel of $V_{\Ome_\ep}^{(\al)}$ guarantees that each interval has an $r^{\al -m}$-center. In particular, the potential $V_{\Ome_\ep}^{(\al)}$ has at least two maximizers.
}
\end{ex}

\begin{ex}\label{annulusV}
{\rm Let $\al \leq 0$, and $\Ome = B_3(0) \sm \c{B}_1 (0)$. We take an open cone $C$ of aperture angle $\kappa$ and height $\de$ such that the complement of $\Ome$ satisfies the uniform boundary inner cone condition for $C$. We remark $\diam \Ome =6$ and $R_\infty (\Ome )=1$. Let $\tilde{R} = \tilde{R} (\al ,\kappa, \de ,6, 1)$ be as in Lemma \ref{E}. 

Since $Uf(\Ome )= B_2 (0)$ and $\tilde{R} < 1$, Corollary \ref{centerV} implies that any $r^{\al -m}$-center belongs to the annulus $B_{2}(0) \sm \c{B}_{1+\tilde{R}} (0)$. Radial symmetry of the kernel of $V_\Ome^{(\al)}$ guarantees the existence of a positive constant $1+\tilde{R} \leq \rho \leq 2$ such that the set of $r^{\al -m}$-centers contains the sphere $\rho S^{m-1}$.
}
\end{ex}
\section{Estimation of a potential with a summable kernel}
Let $\Ome$ be a body (the closure of a bounded open set) in $\R^m$. In this section, we estimate a potential of the form
\begin{equation}\label{Kalpha}
K_\Ome^{(\al)}(x,t) =\int_\Ome k_\al\( \lvert x-\xi \rvert ,t\)d\xi ,\ x \in \R^m ,\ t>0.
\end{equation}

\begin{ass}\label{assumption}
{\rm For the kernel $k_\al$ in \eqref{Kalpha}, we assume some or all of the following conditions:
\begin{enumerate}[(1)]
\item $k_\al (\cdot ,t)$ is strictly decreasing and satisfies the condition $(C^0_\be)$ for some $\be >0$.
\item We can choose a pair of positive functions $(\psi ,\bar{k}_\al)$ such that the kernel $k_\al$ is expressed as $k(r,t)=\psi (t) \bar{k}_\al (r,t)$, and that $\bar{k}_\al(r,t)$ converges to $r^{\al -m}$ for each positive $r$ as $t$ tends to $0^+$.
\item For each positive $t$, we have
\[
\int_{\R^m} k_\al \( \lvert \xi \rvert ,t\) d\xi =1 .
\]
\item For any positive $\rho$, we have
\[
\lim_{t\to 0^+} \int_{\R^m \sm B_\rho (0)} k_\al \( \lvert \xi \rvert ,t\) d\xi =0 .
\]
\end{enumerate}
Usually, a radially symmetric non-negative kernel is said to be {\it summable} if it satisfies the conditions (3) and (4). 
}
\end{ass}

\begin{lem}\label{estimationK11}
Let $\al \leq 0$. Suppose that $k_\al$ satisfies the conditions $(1)$ and $(2)$ in Assumption \ref{assumption}. Let $X$ and $Y$ be bodies in $\R^m$. Suppose that the complement of $Y$ satisfies the uniform boundary inner cone condition of aperture angle $\kappa$ and height $\de$. Let $R_0 >0$, and $\tilde{R} = \tilde{R} ( \al ,\kappa , \de ,\diam Y, R_0 )$ be given in Lemma \ref{E}. For any $0<b<1$, there exists a positive $\tau_1$ such that if $0<t<\tau_1$, then, for any $x \in X$ with $\dist (x,X^c) \geq R_0$ and $y \in Y$ with $(b/2) \tilde{R} \leq \dist (y,Y^c) \leq b\tilde{R}$, we have $K_Y^{(\al)} (y,t) < K_X^{(\al)}(x,t)$.
\end{lem}

\begin{proof}
If $R_0$ is greater than half of the diameter of $Y$, then the statement obviously holds. Let us assume that $R_0$ is not greater than half of the diameter of $Y$.

Thanks to the uniform boundary inner cone condition of the complement of $Y$, in the same manner as in Theorem \ref{estimationV}, for any point $y \in Y$, we can choose a constant $0\leq \theta (y) \leq (\pi -\kappa )/2$ such that, for each $t$, we have
\[
K_Y^{(\al )}(y,t) < K_{B_{\diam Y}(0) \sm C_{\theta (y)}\( \dist \( y, Y^c \) e_1\)}^{(\al)} (0) ,
\]
where $C_{\theta (y)}\( \dist (y, Y^c) e_1\) = C_{\theta (y)}\( \dist (y, Y^c) e_1;\kappa ,\de \)$ is the cone defined in \eqref{cone1}.

By the assumption for the kernel $k_\al$ and the compactness of the body $Y$, there exits a positive constant $\tau_1$ such that if $0< t< \tau_1$, then, for any $\xi \in (C( (b/2) \tilde{R} e_1; \pi , +\infty )\cap B_{\diam Y}(0)) \cup (B_{\diam Y}(0) \sm B_{R_0}(0) )$, we have
\begin{align*}
&\lvert k_\al \( \lvert \xi \rvert ,t\) -\lvert \xi \rvert^{\al -m} \rvert \\
&< \frac{E \( b \tilde{R} \)}{2 \( \Vol \( C \( \frac{b}{2} \tilde{R} e_1; \pi  , +\infty \) \cap B_{\diam Y}(0) \) + \Vol \(B_{\diam Y}(0) \sm B_{R_0}(0)  \)\)} ,
\end{align*}
where $E= E( \cdot ; \al ,\kappa ,\de ,\diam Y, R_0)$ is defined in Lemma \ref{E}. Since, for any $y \in Y$ with $(b/2) \tilde{R} \leq \dist (y,Y^c)$, the cone $C_{\theta (y)} ( \dist (y,Y^c) e_1)$ is contained in the half space $C ((b/2) \tilde{R}e_1 ; \pi  ,+\infty )$, we obtain
\begin{align*}
&\lvert \( \int_{C_{\theta (y)} \( \dist \( y,Y^c \) e_1 \) \cap B_{\diam Y}(0)} - \int_{B_{\diam Y}(0) \sm B_{R_0}(0)} \) \( \bar{k}_\al \( \lvert \xi \rvert ,t\) - \lvert \xi \rvert^{\al -m} \) d\xi \rvert  \\
&\leq  \( \int_{C_{\theta (y)} \( \dist \( y,Y^c \) e_1\) \cap B_{\diam Y}(0)} + \int_{B_{\diam Y}(0) \sm B_{R_0}(0)} \) \lvert \bar{k}_\al \( \lvert \xi \rvert ,t\) - \lvert \xi \rvert^{\al -m} \rvert d\xi \\
&< \frac{1}{2} E \( b \tilde{R} \),
\end{align*} 
for any $y \in Y$ with $(b/2) \tilde{R} \leq \dist (y,Y^c)$.

Hence if $0<t<\tau_1$, then, for any $x \in X$ with $\dist (x,X^c) \geq R_0$ and $y \in Y$ with $(b/2) \tilde{R} \leq \dist (y,Y^c) \leq b\tilde{R}$, we obtain
\begin{align*}
&\frac{K_X^{(\al )} (x,t) -K_Y^{(\al )} (y,t)}{\psi (t)} \\
&> \( \int_{C_{\theta (y)} \( \dist \( y,Y^c \) e_1 \) \cap B_{\diam Y}(0)} - \int_{B_{\diam Y}(0) \sm B_{R_0}(0)} \) \bar{k}_\al \( \lvert \xi \rvert ,t \) d\xi \\
&>  \( \int_{C_{\theta (y)} \( \dist \( y,Y^c \) e_1 \) \cap B_{\diam Y}(0)} - \int_{B_{\diam Y}(0) \sm B_{R_0}(0)} \) \lvert \xi \rvert^{\al-m} d\xi  -\frac{1}{2} E \( b\tilde{R}  \) \\
&\geq E \( \dist \( y,Y^c \) \) -\frac{1}{2} E \( b\tilde{R} \) \\
&\geq \frac{1}{2} E \( b\tilde{R} \) \\
&>0 ,
\end{align*}
where the forth inequality follows from the first assertion in Lemma \ref{E}.
\end{proof}

\begin{lem}\label{estimationK12}
Let $\al$, $k_\al$, $X$, $Y$, $R_0$, $\tilde{R}$, $b$ and $\tau_1$ be as in Lemma \ref{estimationK11}. If $0< t< \tau_1$, then, for any $x \in X$ with $\dist (x,X^c) \geq R_0$ and $y \in Y$ with $\dist (y,Y^c) \leq (b/2) \tilde{R}$, we have $K_Y^{(\al)} (y,t) < K_X^{(\al)}(x,t)$.
\end{lem}

\begin{proof}
For any $0 \leq \theta \leq (\pi -\kappa )/2$ and $0\leq R \leq (b/2)\tilde{R}$, the strictly decreasing behavior of $k_\al (\cdot ,t)$ implies
\begin{align*}
&\( \int_{C_\theta \( R e_1 \) \cap B_{\diam Y}(0)} - \int_{B_{\diam Y}(0) \sm B_{R_0} (0)} \) k_\al \( \lvert \xi \rvert ,t \) d\xi \\
&> \( \int_{C_\theta \( \( R +(b/2) \tilde{R} \) e_1 \) \cap B_{\diam Y}(0)} - \int_{B_{\diam Y}(0) \sm B_{R_0} (0)} \) k_\al \( \lvert \xi \rvert ,t \) d\xi \\
&\geq \frac{\psi (t)}{2} E \( b \tilde{R} \) ,
\end{align*}
where the last inequality was shown in Lemma \ref{estimationK11}. This inequality implies the conclusion in the same manner as in Lemma \ref{estimationK11}.
\end{proof}

\begin{prop}\label{estimationK1}
Let $\al$, $k_\al$, $X$, $Y$, $R_0$, $\tilde{R}$, $b$ and $\tau_1$ be as in Lemma \ref{estimationK11}. If $0< t< \tau_1$, then, for any $x \in X$ with $\dist (x,X^c) \geq R_0$ and $y \in Y$ with $\dist (y,Y^c ) \leq b\tilde{R}$, we have $K_Y^{(\al )} (y,t) < K_X^{(\al )} (x,t)$.
\end{prop}

\begin{lem}\label{small-time_cone} 
Suppose that $k_\al$ satisfies the conditions $(3)$ and $(4)$ in Assumption \ref{assumption}. Let $C_\theta (x;\kappa ,\de)$ be the cone defined in \eqref{cone1}. For any $\theta$, we have
\[
\lim_{t\to 0^+} \int_{C_\theta (0;\kappa ,\de)} k_\al \( \lvert \xi \rvert ,t\) d\xi = \frac{\sigma \( C(0;\kappa ,1) \cap S^{m-1}\)}{\sigma\(S^{m-1}\)} .
\]
\end{lem}

\begin{proof}
We remark that the conditions (3) and (4) imply
\[
\lim_{t\to 0^+} \int_{B_\de(0)} k_\al \( \lvert \xi \rvert ,t\) d\xi 
= \lim_{t\to 0^+} \( \int_{\R^m} -\int_{\R^m \sm B_\de (0)} \) k_\al \( \lvert \xi \rvert ,t\) d\xi 
=1 .
\]

Since we have 
\[
\Vol \( C_\theta \( 0;\kappa ,\de \) \) = \frac{\sigma \( C(0;\kappa ,1) \cap S^{m-1}\)}{\sigma\(S^{m-1}\)}  \Vol \( B_\de (0)\) ,
\]
the rotation invariance of our potential implies
\[
\int_{C_\theta (0;\kappa ,\de)} k_\al \( \lvert \xi \rvert ,t\) d\xi 
= \frac{\sigma \( C(0;\kappa ,1) \cap S^{m-1}\)}{\sigma\(S^{m-1}\)} \int_{B_\de (0)} k_\al \( \lvert \xi \rvert ,t\) d\xi
\to \frac{\sigma \( C(0;\kappa ,1) \cap S^{m-1}\)}{\sigma\(S^{m-1}\)} 
\]
as $t$ tends to $0^+$.
\end{proof}

\begin{prop}\label{estimationK2} 
Suppose that $k_\al$ satisfies the conditions $(3)$ and $(4)$ in Assumption \ref{assumption}. Let $X$ and $Y$ be bodies in $\R^m$. Suppose that the complement of $Y$ satisfies the uniform interior cone condition of aperture angle $\kappa$ and height $\de$. Let $R_0 >0$.  There exists a positive $\tau_2$ such that if $0<t<\tau_2$, then, for any $x \in X$ with $\dist (x,X^c) \geq R_0$ and $y \in Y^c$, we have $K_Y^{(\al)}(y,t) < K_X^{(\al)}(x,t)$.
\end{prop}

\begin{proof}
By the conditions (3) and (4) for the kernel, we can choose a positive constant $\tau_{21}$ such that if $0<t<\tau_{21}$, then, for any point $x \in X$ with $\dist (x,X^c) \geq R_0$, we have
\[
K_X^{(\al)}(x,t) \geq K_{B_{R_0}(0)}^{(\al)} (0,t) > 1-\frac{\sigma \( C(0;\kappa ,1) \cap S^{m-1} \)}{2\sigma \( S^{m-1}\)} ,
\]
where the cone $C(0;\kappa ,1)$ is defined in \eqref{cone0}.

On the other hand, we can choose a positive constant $\tau_{22}$ such that if $0<t<\tau_{22}$, then, for any $y \in Y^c$, the uniform interior cone condition of $Y^c$ and Lemma \ref{small-time_cone} imply
\[
K_Y^{(\al)} (y,t) \leq K_{\R^m \sm C(0)}^{(\al)}(0,t) =1- K_{C(0)}^{(\al)}(0,t) <1-\frac{\sigma \( C(0;\kappa ,1) \cap S^{m-1} \)}{2\sigma \( S^{m-1}\)} .
\]

Taking $\tau_2 =\min \{ \tau_{21},\ \tau_{22} \}$, the proof is completed.
\end{proof}

\begin{thm}\label{estimationK} 
Let $\al \leq 0$. Suppose that $k_\al$ satisfies all the conditions in Assumption \ref{assumption}. Let $X$ and $Y$ be bodies in $\R^m$. Suppose that the complement of $Y$ satisfies the uniform interior cone condition of aperture angle $\kappa$ and height $\de$. Let $R_0>0$, and $\tilde{R} = \tilde{R} (\al ,\kappa ,\de ,\diam Y, R_0 )$ be given in Lemma \ref{E}. For any $0<b<1$, there exists a positive $\tau$ such that if $0 < t<\tau$, then, for any $x \in X$ with $\dist (x,X^c)\geq R_0$ and $y \in \R^m$ with $\dist (y,Y^c) \leq b\tilde{R}$, we have $K_Y^{(\al)}(y,t) < K_X^{(\al)}(x,t)$.
\end{thm}

\begin{proof}
Thanks to Lemma \ref{cone_condition}, the complement of $Y$ satisfies the uniform boundary inner cone condition of aperture angle $\kappa$ and height $\de$. Let $\tau_1$ and $\tau_2$ be as in Propositions \ref{estimationK1} and \ref{estimationK2}, respectively. Taking $\tau =\min \{ \tau_1,\ \tau_2 \}$, we obtain the conclusion.
\end{proof}

\begin{cor}\label{centerK}
Let $\al \leq 0$ and $k_\al$ be as in Theorem \ref{estimationK}. Let $\Ome$ be a body in $\R^m$ whose complement satisfies the uniform interior cone condition of aperture angle $\kappa$ and height $\de$. Let $\tilde{R} = \tilde{R} ( \al ,\kappa, \de , \diam \Ome ,R_\infty (\Ome ))$ be given in Lemma \ref{E}. For any $0<b<1$, there exits a positive constant $\tau$ such that if $0<t<\tau$, then any $k_\al$-center of $\Ome$ at time $t$ is contained in the intersection $(\Ome \sim b\tilde{R}B^m) \cap Uf(\Ome )$.
\end{cor}

\begin{proof}
Thanks to Proposition \ref{ufK}, all the $k_\al$-centers are contained in the minimal unfolded region of $\Ome$. Hence, combining Theorem \ref{estimationK}, we obtain the conclusion.
\end{proof}

\begin{ex}\label{dumbbellK}
{\rm Let $\al \leq 0$. Suppose that $k_\al$ satisfies all the conditions in Assumption \ref{assumption}. Let $\ep$, $\Ome_\ep$, $C$ and $\tilde{R}$ be as in Example \ref{dumbbellV}. Fix an $0< \ep < \tilde{R}$. 

Corollary \ref{centerK} guarantees the existence of a positive constant $\tau$ such that if $0<t <\tau$, then any $k_\al$-center of the body $\Ome_\ep$ at time $t$ belongs to the disjoint union of the intervals $( [ -2 ,-1 -\sqrt{\tilde{R}^2-\ep^2} ] \cup [ 1+ \sqrt{\tilde{R}^2-\ep^2} ,2 ] ) \times \{ 0 \}^{m-1}$. Radial symmetry of the kernel of $K_{\Ome_\ep}^{(\al)} (\cdot ,t)$ guarantees that each interval has an $k_\al$-center. In particular, the potential $K_{\Ome_\ep}^{(\al)}(\cdot ,t)$ has at least two maximizers for any sufficiently small $t$.
}
\end{ex}

\begin{ex}\label{annulusK}
{\rm Let $\al \leq 0$. Suppose that $k_\al$ satisfies all the conditions in Assumption \ref{assumption}. Let $\Ome$, $C$ and $\tilde{R}$ be as in Example \ref{annulusV}. 

Corollary \ref{centerK} guarantees the existence of a positive constant $\tau$ such that if $0<t<\tau$, then any $k_\al$-center of $\Ome$ at time $t$ belongs to the annulus $B_{2}(0) \sm \c{B}_{1+\tilde{R}} (0)$. Radial symmetry of the kernel of $K_\Ome^{(\al)}(\cdot , t)$ guarantees the existence of a positive constant $1+\tilde{R} \leq \rho (t) \leq 2$ such that the set of $k_\al$-centers of $\Ome$ at time $t$ contains the sphere $\rho (t) S^{m-1}$ for any sufficiently small $t$.
}
\end{ex}

\begin{cor}\label{uniqueK} 
Let $\al \leq 0$ and $k_\al$ be as in Theorem \ref{estimationK}. Let $\Ome$ be a convex body in $\R^m$. Let $\tilde{R} = \tilde{R} (\al , \pi , +\infty , \diam \Ome , R_\infty (\Ome ))$ be given in Lemma \ref{E}. Let $0<b<1$, and $\Ome'= ( \Ome \sim b\tilde{R}B^m) \cap Uf(\Ome )$. Suppose the existence of a positive constant $\tau'$ such that, for any $0<t <\tau'$, $k_\al(r,t)r^{m-1}$ is decreasing on the interval $[d(\Ome ,\Ome'),D(\Ome ,\Ome')]$ with respect to $r$. There exists a positive constant $\tau \leq \tau'$ such that, for any $0<t <\tau$, $K_\Ome^{(\al )} (\cdot ,t)$ is strictly concave on $\Ome'$. In particular, $\Ome$ has a unique $k_\al$-center at time $0< t <\tau$.
\end{cor}

\begin{proof}
Since $(\Ome \sim b\tilde{R}B^m) \cap Uf(\Ome )$ is convex and contained in the interior of $\Ome$, Propositions \ref{concavity} guarantees the conclusion.
\end{proof}

\begin{thm}\label{behaviorK1}
Let $\al \leq 0$. Suppose that $k_\al$ satisfies all the conditions in Assumption \ref{assumption}. Let $\Ome$ be a body in $\R^m$ whose complement satisfies the uniform interior cone condition of aperture angle $\kappa$ and height $\de$. For any decreasing sequence $\{ t_\ell \}$ with zero limiting value and any $k_\al$-center $c_\al (t_\ell )$ at time $t_\ell$, the distance between $c_\al (t_\ell )$ and the set of $r^{\al -m}$-centers tends to zero as $\ell$ goes to $+\infty$.
\end{thm}

\begin{proof}
Thanks to Corollary \ref{centerK}, we may assume that any $k_\al$-center at time $t_\ell$ belongs to the inner-parallel body of $\Ome$ of radius $(1/2)\tilde{R}$, where $\tilde{R} = \tilde{R} (\al , \kappa ,\de ,\diam \Ome , R_\infty (\Ome))$ is given in Lemma \ref{E}. Since the inner-parallel body is compact, without loss of generality, we assume that $\{ c_\al (t_\ell)\}$ converges to a point $c_\al$. In order to show that $c_\al$ is an $r^{\al -m}$-center of $\Ome$, we assume that $c_\al$ is not any $r^{\al-m}$-center, and let us derive a contradiction.

Fix an arbitrary $0<\ep <  (1/2)\tilde{R}$. Then, for any point $x$ in the inner-parallel body of $\Ome$ of radius $(1/2)\tilde{R}$, we have
\begin{align*}
K_\Ome^{(\al)} (x,t) 
&=\( \int_{\Ome \sm B_\ep (x)} + \int_{B_\ep(x)}\) k_\al \( \lvert x -\xi\rvert ,t\) d\xi \\
&=\int_{\Ome \sm B_\ep (x)} k_\al \( \lvert x -\xi \rvert,t\) d\xi +\sigma \( S^{m-1} \) \int_0^\ep k_\al \( r,t\) r^{m-1}dr .
\end{align*}
Therefore, the maximum value of $K_\Ome^{(\al)}(\cdot,t_\ell)$ is attained at $c_\al (t_\ell)$ if and only if that of the function
\[
\Ome \sim  \frac{1}{2}\tilde{R}B^m \ni x \mapsto \int_{\Ome \sm B_\ep (x)} \bar{k}_\al \( \lvert x -\xi \rvert,t_\ell \) d\xi \in \R
\]
is attained at $c_\al (t_\ell)$.

Let $p$ be an $r^{\al-m}$-center of $\Ome$. Thanks to the first and second conditions in Assumption \ref{assumption}, there exists a large natural number $L$ such that, for any $\ell \geq L$, the following inequalities hold:
\begin{align*}
\lvert \int_{\Ome \sm B_\ep \( c_\al \( t_\ell \)\)} \bar{k}_\al \( \lvert c_\al \( t_\ell \) -\xi \rvert ,t_\ell \) d\xi 
-\int_{\Ome \sm B_\ep \( c_\al \)} \lvert c_\al -\xi \rvert^{\al-m}d\xi \rvert
\leq \frac{V_\Ome^{(\al)}(p) -V_\Ome^{(\al)} \( c_\al\)}{3} ,\\
\lvert \int_{\Ome \sm B_\ep (p)} \bar{k}_\al \( \lvert p-\xi \rvert ,t_\ell \) d\xi 
-\int_{\Ome \sm B_\ep (p)} \lvert p -\xi \rvert^{\al-m}d\xi \rvert
\leq \frac{V_\Ome^{(\al)}(p) -V_\Ome^{(\al)} \( c_\al\)}{3} .
\end{align*}
Hence, using Proposition \ref{without_limit}, we obtain
\begin{align*}
0
&\leq \int_{\Ome \sm B_\ep \( c_\al \( t_\ell \)\)} \bar{k}_\al \( \lvert c_\al \( t_\ell \) -\xi \rvert ,t_\ell \) d\xi 
-\int_{\Ome \sm B_\ep (p)} \bar{k}_\al \( \lvert p-\xi \rvert ,t_\ell \) d\xi  \\
&<\( \int_{\Ome \sm B_\ep \( c_\al \)} \lvert c_\al -\xi \rvert^{\al-m}d\xi +\frac{V_\Ome^{(\al)}(p) -V_\Ome^{(\al)} \( c_\al\)}{3} \) \\
&\quad -\( \int_{\Ome \sm B_\ep (p)} \lvert p -\xi \rvert^{\al-m}d\xi - \frac{V_\Ome^{(\al)}(p) -V_\Ome^{(\al)} \( c_\al\)}{3}\) \\
&=- \frac{V_\Ome^{(\al)}(p) -V_\Ome^{(\al)} \( c_\al\)}{3} \\
&<0,
\end{align*}
which is a contradiction.
\end{proof}

\begin{cor}\label{behaviorK2} 
Let $\al \leq 0$ and $k_\al$ be as in Theorem \ref{behaviorK1}. Let $\Ome$ be a convex body. The set of $k_\al$-centers at time $t$ converges to the set of $r^{\al -m}$-centers as $t$ tends to $0^+$ with respect to the Hausdorff distance.
\end{cor}

\begin{proof}
Theorem \ref{uniquenessV} guarantees the uniqueness of an $r^{\al-m}$-center of $\Ome$. Hence Theorem \ref{behaviorK1} implies the conclusion.
\end{proof}
\section{Applications to the Poisson integral}
Let $\Ome$ be a body (the closure of a bounded open set) in $\R^m$. In this section, we apply the results in the previous section to the Poisson integral for the upper half-space. In other words, we consider the small-height behavior of illuminating centers of a body. 

For the Poisson integral, the kernel in \eqref{Kalpha} is give by
\begin{equation}\label{kernelA}
k_{-1}(r,h)
=\psi(h) \bar{k}_{-1}(r,h) 
=\frac{2h}{\sigma_m \( S^m \)} \( r^2 +h^2\)^{-(m+1)/2} .
\end{equation}
From the facts \eqref{total_angle} and \eqref{local_angle}, the kernel \eqref{kernelA} exactly satisfies the conditions in Assumption \ref{assumption}.

\begin{prop}\label{estimationA1} 
Let $X$ and $Y$ be bodies in $\R^m$. Suppose that the complement of $Y$ satisfies the uniform boundary inner cone condition of aperture angle $\kappa$ and height $\de$. Let $R_0 >0$, and $\tilde{R}=\tilde{R} (-1, \kappa , \de ,\diam Y , R_0)$ be given in Lemma \ref{E}. For any $0<b<1$, there exists a positive $h_1$ such that if $0<h<h_1$, then, for any $x \in X$ with $\dist (x,X^c)\geq R_0$ and $y \in Y$ with $\dist (y,Y^c) \leq b\tilde{R}$, we have $A_Y(y,h) < A_X(x,h)$.

{\rm (This fact follows from Proposition \ref{estimationK1}.)}
\end{prop}

\begin{prop}\label{estimationA2} 
Let $X$ and $Y$ be bodies in $\R^m$. Suppose that the complement of $Y$ satisfies the uniform interior cone condition of aperture angle $\kappa$ and height $\de$. Let $R_0>0$. There exists a positive $h_2$ such that if $0<h<h_2$, then, for any $x \in X$ with $\dist (x,X^c)\geq R_0$ and $y \in Y^c$, we have $A_Y(y,h) < A_X(x,h)$.

{\rm (This fact follows from Proposition \ref{estimationK2}.)}
\end{prop}

\begin{prop}\label{estimationA} 
Let $X$ and $Y$ be bodies in $\R^m$. Suppose that the complement of $Y$ satisfies the uniform interior cone condition of aperture angle $\kappa$ and height $\de$. Let $R_0>0$, and $\tilde{R} = \tilde{R} (-1, \kappa ,\de ,\diam Y, R_0)$ be given in Lemma \ref{E}. For any $0<b<1$, there exists a positive $h_0$ such that if $0<h<h_0$, then, for any $x \in X$ with $\dist (x,X^c)\geq R_0$ and $y \in \R^m$ with $\dist (y,Y^c) \leq b\tilde{R}$, we have $A_Y(y,h) < A_X(x,h)$.

{\rm (This fact follows from Theorem \ref{estimationK}.)}
\end{prop} 

\begin{cor}\label{centerA} 
Let $\Ome$ be a body in $\R^m$ whose complement satisfies the uniform interior cone condition of aperture angle $\kappa$ and height $\de$. Let $\tilde{R} =\tilde{R} (-1, \kappa ,\de ,\diam \Ome ,R_\infty (\Ome ) )$ be given in Lemma \ref{E}. For any $0<b<1$, there exists a positive $h_0$ such that, for any $0< h<h_0$, any illuminating center of $\Ome$ of height $h$ is contained in the intersection $(\Ome \sim b\tilde{R}B^m) \cap Uf(\Ome )$.

{\rm (This fact follows from Corollary \ref{centerK}.)}
\end{cor}

\begin{ex}\label{dumbbellA}
{\rm Let $\ep$, $\Ome_\ep$ and $C$ be as in Example \ref{dumbbellV}. Let $\tilde{R}= \tilde{R} ( -1 ,\kappa ,\de ,2\sqrt{10} ,1)$ be as in Lemma \ref{E}, and fix an $0< \ep < \tilde{R}$. 

Corollary \ref{centerA} guarantees the existence of a positive constant $h_0$ such that if $0<h <h_0$, then any illuminating center of the body $\Ome_\ep$ of height $h$ belongs to the disjoint union of the intervals $( [ -2 ,-1 -\sqrt{\tilde{R}^2-\ep^2} ] \cup [ 1+ \sqrt{\tilde{R}^2-\ep^2} ,2 ] ) \times \{ 0 \}^{m-1}$. Radial symmetry of the Poisson kernel guarantees that each interval has an illuminating center. In particular, the Poisson integral $P_{\Ome_\ep}(\cdot ,h)$ has at least two maximizers for any sufficiently small $h$.
}
\end{ex}

\begin{ex}\label{annulusA}
{\rm Let $\Ome$ and $C$ be as in Example \ref{annulusV}. Let $\tilde{R} = \tilde{R} (-1 ,\kappa, \de ,6, 1)$ be as in Lemma \ref{E}. 

Corollary \ref{centerA} guarantees the existence of a positive constant $h_0$ such that if $0<h<h_0$, then any illuminating center of $\Ome$ of height $h$ belongs to the annulus $B_{2}(0) \sm \c{B}_{1+\tilde{R}} (0)$. Radial symmetry of the Poisson kernel guarantees the existence of a positive constant $1+\tilde{R} \leq \rho (h) \leq 2$ such that the set of illuminating centers of $\Ome$ of height $h$ contains the sphere $\rho (h) S^{m-1}$ for any sufficiently small $h$.
}
\end{ex}

\begin{cor}\label{uniqueA} 
Let $\Ome$ be a convex body in $\R^m$. Let $\tilde{R} = \tilde{R} (-1, \pi ,+\infty ,\diam \Ome ,R_\infty (\Ome ))$ be given in Lemma \ref{E}. Let $0< b<1$, and $\Ome'= ( \Ome \sim b \tilde{R} B^m ) \cap Uf(\Ome )$. There exists a positive constant $h_0$ such that, for any $0< h < h_0$, the Poisson integral $P_\Ome (\cdot ,h)$ is strictly concave on $\Ome'$. In particular, $\Ome$ has a unique illuminating center of height $0<h<h_0$.
\end{cor}

\begin{proof}
We can directly show that if $h \leq \sqrt{(m-1)/2} d( \Ome, \Ome' )$, then the function $r \mapsto (r^2+h^2)^{-(m+1)/2} r^{m-1}$ is decreasing for $d( \Ome, \Ome') \leq r \leq D( \Ome, \Ome')$. Let $h_0' = \sup \{ h>0 \vert \mathcal{A}_\Ome (h') \subset \Ome' \ \forall h' <h \}$. Taking $h_0= \min\{ \sqrt{(m-1)/2} d( \Ome, \Ome' ) ,h_0' \}$, Corollary \ref{uniqueK} implies the conclusion.
\end{proof}

\begin{prop}\label{behaviorA1} 
Let $\Ome$ be a body in $\R^m$ whose complement satisfies the uniform interior cone condition of aperture angle $\kappa$ and height $\de$. For any decreasing sequence $\{ h_\ell \}$ with zero limiting value and any illuminating center $c(h_\ell)$ of height $h_\ell$, the distance between $c(h_\ell)$ and the set of $r^{-(m+1)}$-centers tends to zero as $\ell$ goes to $+\infty$.

{\rm (This fact follows from Theorem \ref{behaviorK1}.)}
\end{prop}

\begin{cor}\label{behaviorA2} 
Let $\Ome$ be a convex body in $\R^m$. The set of illuminating centers of height $h$ converges to the set of $r^{-(m+1)}$-centers as $h$ tends to $0^+$ with respect to the Hausdorff distance.

{\rm (This fact follows from Corollary \ref{behaviorK2}.)}
\end{cor}
\section{Appendix: A lower bound of {\boldmath $\tilde{R}(-1, \pi ,+\infty , \diam \Ome ,R_0)$}}
Let $\Ome$ be a convex body in $\R^m$. Thanks to the convexity of $\Ome$, we can take the uniform boundary inner cone of the complement of $\Ome$ as a half space. Let $0< R_0 < \diam \Ome$. In this appendix, we give a lower bound of $\tilde{R} = \tilde{R}(-1, \pi, +\infty ,\diam \Ome ,R_0)$. Let us estimate the zero-point of the function
\begin{equation}
E (R) = \( \int_{C\( Re_1; \pi ,+\infty \) \cap B_{\diam \Ome}(0)} -\int_{B_{\diam \Ome}(0) \sm B_{R_0}(0)} \) \lvert \xi \rvert^{-(m+1)} d \xi .
\end{equation}

Let $\f (R)=\arccos (R/\diam \Ome)$. Using the polar coordinate, we obtain 
\begin{align}
\nonumber
E(R)
&=\sigma_{m-2} \( S^{m-2}\) \int_0^{\f(R)} \( \int_{R/\cos \theta}^{\diam \Ome} r^{-2}dr \) \sin^{m-2}\theta d\theta\\
\nonumber
&\quad -\sigma_{m-2} \( S^{m-2}\)\int_0^\pi \( \int_{R_0}^{\diam \Ome} r^{-2} dr\) \sin^{m-2}\theta d\theta\\
\nonumber
&=\frac{\sigma_{m-2} \( S^{m-2}\)}{R} \( \frac{\sin^{m-1} \f(R)}{m-1} +\frac{R}{\diam \Ome} \int_{\f(R)}^\pi \sin^{m-2}\theta d\theta -\frac{R}{R_0}\int_0^\pi \sin^{m-2}\theta d\theta \) \\
&=: \frac{\sigma_{m-2} \( S^{m-2}\)}{R}f(R).
\end{align}
Direct computation shows the following properties:
\begin{align}
f'(R)&=\frac{1}{\diam \Ome} \int_{\f(R)}^\pi \sin^{m-2}\theta d\theta -\frac{1}{R_0} \int_0^{\pi} \sin^{m-2}\theta d\theta <0,\\
f''(R)&=-\f'(R) \sin^{m-2} \f (R) >0.
\end{align}
Since we have
\begin{equation}
f'(0)= \( \frac{1}{\diam \Ome} -\frac{2}{R_0} \) \int_0^{\pi/2} \sin^{m-2}\theta d\theta ,
\end{equation}
we obtain 
\begin{align}
\nonumber
\tilde{R} \( -1, \pi ,+\infty ,\diam \Ome , R_0\)
&> -\frac{f(0)}{f'(0)}\\
\nonumber
&=\frac{1}{\ds (m-1)\( \frac{2}{R_0}-\frac{1}{\diam \Ome} \) \int_0^{\pi/2} \sin^{m-2}\theta d\theta}\\
&\geq \frac{R_0}{\ds 2(m-1)\int_0^{\pi/2} \sin^{m-2}\theta d\theta} .
\end{align}
For example, in the case of $m=2$, the above lower bound coincides with $R_0 / \pi \approx 0.3183 R_0$.

\no 
Faculty of Education and Culture,\\
University of Miyazaki,\\
1-1, Gakuen Kibanadai West, Miyazaki city, Miyazaki prefecture, 889-2155, Japan\\
E-mail: sakata@cc.miyazaki-u.ac.jp
\end{document}